\documentclass[12pt]{article}

\usepackage{amsmath,amsfonts,amssymb,amsthm,mathrsfs}
\usepackage{graphicx}
\usepackage{stmaryrd}
\usepackage{hyperref}
\usepackage{tikz}

\theoremstyle{definition}
\newtheorem{definition}{Definition}
\newtheorem{remark}[definition]{Remark}
\newtheorem{example}[definition]{Example}

\theoremstyle{mytheorem}
\newtheorem{theorem}[definition]{Theorem}
\newtheorem{lemma}[definition]{Lemma}
\newtheorem{proposition}[definition]{Proposition}

\newtheorem{corollary}[definition]{Corollary}
\newtheorem{assumption}[definition]{Assumption}

\usepackage{color}

\renewcommand{\th}{u }

\newcommand{\equlaw}{\stackrel{(d)}{=}}

\renewcommand{\P}{\mathbf{P}} 
\newcommand{\E}{\mathbf{E}}

\renewcommand{\S}{\mathcal{S}}
\newcommand{\Per}{{\rm Per}}

\renewcommand{\u}{\mathbf{u}}

\newcommand{\N}{\mathscr{N}}

\newcommand{\sign}{\text{\rm{sign}}}

\newcommand{\supp}{\text{\rm{supp}}}

\newcommand{\C} {\mathcal{C} }

\newcommand{\EC}{Euler characteristic}
 \newcommand{\notessential}[1]{#1} 
\def\mode{1}
\def\mode{3} 
  \if\mode3 
   
 \renewcommand{\notessential}[1]{}
 \fi
\newcommand{\EP}{Euler primitive}

\newlength{\querylen}
\usepackage{fancybox}

\numberwithin{equation}{section}
\numberwithin{definition}{section}
\title{Two-dimensional Kac-Rice formula. Application to shot noise processes excursions.}
\author{Rapha\"el Lachi\`eze-Rey\footnote{UFR Maths-Info, Universit\'e Paris Descartes, 45 Rue des Saints-P\`eres, 75006 Paris,
raphael.lachieze-rey@parisdescartes.fr}}
\date{} 
\begin{document} 

\maketitle

{\bf Abstract}

Given a deterministic function $f:\mathbb{R}^{2}\to \mathbb{R}$ satisfying suitable assumptions, we show that for $h$ smooth with compact support,
\begin{align*}
\int_{\mathbb{R}}\chi (\{f\geqslant u\})h(u)du=\int_{\mathbb{R}^{2}}\gamma (x,f,h)dx,
\end{align*}where $\chi (\{f\geqslant u\})$ is the \EC~of the excursion set of $f$ above the level $u$, and $\gamma (x,f,h)$ is a bounded function depending on $\nabla f(x)$, $h(f(x))$, $h'(f(x))$ and $\partial _{ii}f(x),i=1,2$. This formula can be seen as a $2$-dimensional analogue of Kac-Rice formula. It yields in particular that the left hand member is continuous in the argument $f$, for an appropriate norm on the space of $\C^{2}$ functions.

If $f$ is a random field, the expectation can be passed under integrals in this identity under minimal requirements, not involving any density assumptions on the marginals of $f$ or his derivatives.
We apply these results to give a weak expression of the mean \EC~of a shot noise process, and the finiteness of its moments.\\

 {\bf Keywods:} Random excursions, \EC, Kac-Rice formula, co-area formula, shot noise processes, Gaussian fields

\section{Introduction}

The Euler characteristic of a set $A\subset \mathbb{R}^{2}$, denoted $\chi (A)$ whenever it is well defined, is a topological invariant used for many purposes. Its additivity and topological properties make it a privileged indicator in   nonparametric spatial statistics \cite{TayWor}, study of random media \cite{SWRS}, it is a privileged topological  index in astronomy \cite{Mar15,Mel90}, brain imagery \cite{KilFri}, or oceanography. See also \cite{ABBSW} and references therein for a general review of applied algebraic topology.   In the last two decades, much interest has been brought to the geometric properties of level sets of  multivariate random fields, and in particular to their \EC~\cite{AdlSam,AST,AdlTay07,BieDes12,BieDes,EstLeo14,LacEC2,Mar15,TayWor}. Given a  function $f:\mathbb{R}^{2}\to \mathbb{R}$, we call \emph{excursion set}, or \emph{upper-level set} of $f$, the set $\{f\geqslant u\}=\{x\in \mathbb{R}^{2}:f(x)\geqslant u\}$, for $u\in \mathbb{R}$.

Most results concerning the mean \EC~of random excursions address Gaussian random fields, as their finite dimensional distributions are easier to handle, and more general results require the field to satisfy strong density requirements. In this paper, we use the recent general variographic approach \cite{LacEC2} 
to give the mean value of a bidimensionnal weak version of the \EC~which does not require density hypotheses.

The method derives from a purely deterministic application of the results  of \cite{LacEC2}.
 Given a sufficiently regular   function $f$ on $\mathbb{R}^{2}$, we consider the \emph{Euler primitive} of $f$, which to   a smooth test function $h:\mathbb{R}\to \mathbb{R}$ associates
\begin{align}
\label{eq:euler-integral}
\chi _{f}(h):=\int_{\mathbb{R}}h(u)\chi (\{f\geqslant u\})du.
\end{align}
We show in this paper that the \EP~can be written as a proper Lebesgue integral over $\mathbb{R}^{2}$,  involving the first and second order derivatives of $f$. Note $\u_{1},\u_{2}$ the canonical basis of $\mathbb{R}^{2}$, $\partial _{i}f$ the partial derivative of $f$ along $\u_{i},i=1,2$, and $\partial ^{2}_{ii}f$ the second order partial derivative in direction $\u_{i}$. Then, Theorem \ref{thm:main} states that if $f$ is Morse and has compact excursion sets, we have 
\begin{align}
\label{eq:main-intro}
\chi _{f}(h)=-\sum_{i=1}^{2}\int_{\mathbb{R}^{2}}\left[
h'(f(x))\partial _{i}f(x)^{2}+h(f(x))\partial^2_{ii}f(x)
\right]\mathbf{1}_{\{\nabla f(x)\in Q_{i}\}}dx,
\end{align}where $Q_{1}$ and $Q_{2}$ are two quarter planes defined at \eqref{eq:Q1Q2}.  The compacity assumption on the upper level set is present mainly to avoid taking care of boundary effects when one intersects an unbounded excursion with a bounded window. Such a study could be done by using the boundary estimates from \cite{LacEC2}, but is beyond the scope of the present paper.

The classical litterature   gives   the Euler characteristic of an excursion in function 
of the indexes of its critical points above the considered level. Namely, for a Morse function $f$ and $u\in \mathbb{R}$ such that $\{f\geqslant u\}$ is compact and does not have critical points on its boundary, we have 
\begin{align}
\label{eq:intro-classical-formula}
\chi (\{f\geqslant u\})=\sum_{k=0}^{2}(-1)^{k}\mu _{k}(f,u), 
\end{align}where 
\begin{align*}
\mu _{k}(f,u)=\#\{x\in  \{f\geqslant u\}:\nabla f(x)=0&\text{\rm{ and the Hessian matrix of $f$ in $x$}} \\
&\text{\rm{has exactly $k$ positive eigenvalues}}\}.
\end{align*} 

In practical uses of this formula in dimension $2$, see e.g. \cite{EstLeo14}, the counting measure on the right hand side is captured through an integral over a neighbourhood of the critical points, see \cite[Theorem 11.2.3]{AdlTay07},
\begin{align*}
\chi (\{f\geqslant u\})=\lim_{\varepsilon \to 0}\frac{1}{4\varepsilon ^{2}}\int_{\mathbb{R}^{2}}\mathbf{1}_{\{\|\nabla f(x)\|_{\infty }\leqslant \varepsilon \}}\mathbf{1}_{\{f(x)\geqslant u\}}\det(H_{f}(x))dx.
\end{align*}
 In contrast, our approach focuses here on the behaviour of $f$ and its derivatives around the boundary of the considered excursion set. In this respect, our weak formulation breaks free from the critical-point theory, and formally requires twice differentiability only almost everywhere, also called $\C^{1,1}$ \emph{regularity}.  
Also, while \eqref{eq:intro-classical-formula} hardly  makes sense without Morse assumptions, only $\C^{1,1}$ regularity is required to define the right-hand member of \eqref{eq:main-intro}. We advocate through Section \ref{sec:non-morse} that the latter might hold in larger generality than the Morse framework, but different ideas seem to be required for a proof. 

Even though the two methods derive the same quantity, it is not clear how to pass directly from \eqref{eq:intro-classical-formula} to  \eqref{eq:main-intro}. The  integral form \eqref{eq:intro-classical-formula} features a more linear structure  than the counting measure concentrated on the critical points. This allows for instance for a better control of the continuity of the Euler characteristic in the argument $f$, see Lemma \ref{ref:lm-continuity}.
With this method, it is also  simpler to establish the finiteness of moments of $\chi _{f}(h)$ for a random field $f$, than through the use of formulas such as \eqref{eq:intro-classical-formula}, see Section \ref{sec:moments}.\\
 
 We exploit \eqref{eq:main-intro} to give the \EP~of a random $\mathcal{C}^{2}$ field $f$ under very broad conditions: almost surely, the critical points of $f$ should be non-degenerate (i.e. $f$ is Morse),   its excursion sets should be compact, and the right hand member of \eqref{eq:main-intro} should have a finite expectation.  The absence of requirement of density for the distribution of $f(x)$ or its derivatives allows us to give the mean value of the \EP~for fields for which it is difficult to have bounds on the density. We consider here the excursion sets of Poisson moving averages, also called shot noise processes. Some related questions have been studied in \cite{AST,BieDes12,BieDes}, but the literature does not contain yet an exact expression for the mean value of the Euler characteristic. We give a weak version of such a formula in Theorems \ref{thm:EC-SN}-\ref{thm:stat-SN}.\\

Formula \eqref{eq:main-intro} can be seen as a two-dimensional analogue of the Kac-Rice formula.
In dimension $1$, the  \EC, noted $\chi ^{(1)}(\cdot )$, is the number of connected components. For the excursion set $\{f\geqslant \th\}$ of a $\C^{1}$ function $f$ with compact level sets, it corresponds to the number of  \emph{up-crossings} at level $\th$, provided $\th$ is not a ctitical value  (see the proof of Theorem \ref{ex:radial}).
The integral version of Kac-Rice formula states that for $h$ smooth with compact support
\begin{align}
\label{eq:kacrice-intro}
\int_{\mathbb{R}}h(u)\chi ^{(1)}(\{f\geqslant u\})du=\int_{\mathbb{R}}h(f(x)) | f'(x) | dx,
\end{align}where the integrand of the left hand member is properly defined for almost all $u$.
 
 Another celebrated formula to which \eqref{eq:main-intro} can be compared, and whom the  formula above is the one-dimensional version, is the co-area formula ((7.4.14) in \cite{AdlTay07}). To introduce this formula in $\mathbb{R}^{2}$,  call \emph{perimeter} of a set $A\subset \mathbb{R}^{2}$, noted $\Per(A),$ the $1$-dimensional Hausdorff measure of its topological boundary.
In $\mathbb{R}^{2}$,
the co-area formula expresses the perimeter of the level sets of a locally Lipschitz function $f:\mathbb{R}^{2}\to \mathbb{R}$, in function of a differential operator applied  to $f$: For $h:\mathbb{R}\to \mathbb{R}$ a bounded measurable function,
\begin{align}
\label{eq:coarea-intro}
\int_{\mathbb{R}}h(u)\Per(\{f\geqslant u\})du=\int_{\mathbb{R}^{2}}h(f(x))\|\nabla f(x)\|dx.
\end{align}
In this respect,  \eqref{eq:main-intro} is an analogue of \eqref{eq:coarea-intro} for the \EC. The perimeter and the Euler characteristic form a remarkable pair of functionals as they are central in the theory of convex bodies. They are, with the volume function, the only homogeneous additive continuous functionals of poly-convex sets of $\mathbb{R}^{2}$, see \cite[Chapter 14]{SchWei}.  In both cases, thanks to  \eqref{eq:main-intro} and \eqref{eq:coarea-intro}, their integral against a test function can be computed in terms of a spatial integral involving $f$ and its derivatives. This gives hope for similar formulas for all additive functionals in higher dimensions.

\section{\EP~}
  Let $m\geqslant 1,W\subset \mathbb{R}^{m}$ measurable.
Let $f$ be a $\C^{1,1}$ function on $W$, i.e. continuously differentiable with Lipschitz gradient. Note $M(f)=\{x\in W:\nabla f(x)=0\}$ its set of \emph{critical points}, and for $\th\in \mathbb{R},$  $M(f ,\th   )=M(f)\cap \{f\geqslant\th  \}$. Also note $V(f)=f(M(f))$ its set of \emph{critical values}. We will make frequent use that by Sard's Theorem, $V(f)$ has $0$ Lebesgue measure.

\begin{definition}Let the previous notation prevail, and assume furthermore that $f$ has compact excursion sets. Let $\th  \in \mathbb{R}\setminus V(f)$. 
Even though the \EC~is not unambiguously defined in all generality, for such a non-degenerate excursion set $\{f\geqslant u\}$ of a $\mathcal{C}^{1,1}$ function, the \EC~corresponds to the number of bounded connected components minus the number of bounded connected components of the complement (also called ``holes''). Depending on the regularity of the set, several equivalent definitions of the \EC~can be adopted, but we use the latter in this paper. 

The function $\th  \in f(W)\cap V(f)^{c} \mapsto \chi (f\geqslant\th  )$ is well defined and measurable, as it is constant on each interval of $f(W)\cap V(f)^{c}$ by \eqref{eq:intro-classical-formula}. Since $V(f)$ is negligible,   we can define, for $h$ a measurable bounded function,
\begin{align*}
\chi_{f}(h)=\int_{\mathbb{R}}h(\th  )\chi (\{f\geqslant \th  \})d\th,
\end{align*}where $\chi (\{f\geqslant \th\})$ takes an arbitrarily irrelevant fixed value, $0$ for instance, on $V(f)$.
This quantity is not well defined in all generality, but in the context where $f$ is Morse, the set of critical points of $f$ is locally finite, and $\chi(\{f\geqslant u\}) $ can only take a finite number of values (see below).
\end{definition}

Say that a $\C^{2}$ function $f:\mathbb{R}^{2}\to \mathbb{R}$ is a \emph{Morse function} if all its critical points are non-degenerate, i.e. if for $x\in M(f),$ the Hessian matrix  
\begin{align*}
H_{f}(x)=\left(
\begin{array}{cc}\partial _{11}^2f(x)&\partial _{12}f(x)\\
\partial _{21}f(x)&\partial _{22}f(x) \end{array}
\right)
\end{align*} of $f$ at  $x$ is non-singular.  Say that $f$ is Morse \textsl{above} some    value $\th \in \mathbb{R}$ if $H_{f}(x)$ is non-singular for $x\in M(f,u)$. In that case, the set of critical points $x\in M(f)$ with $f(x)\geqslant u$ is locally finite. Call furthermore the {\sl index} of $x\in M(f)$ the number of positive eigenvalues of $H_{f}(x)$.

Introduce the quarter-planes
\begin{align}
\label{eq:Q1Q2}
Q_{1}=\{(s,t)\in \mathbb{R}^{2}:t<s<0\},\;
Q_{2}=\{(s,t)\in \mathbb{R}^{2}:s<t<0\}.
\end{align}
Given a  $\C^{1,1}$ function $f$  over some bounded measurable $W\subset \mathbb{R}^{2}$, $f$ is twice differentiable a.e., and for any $\mathcal{C}^{1}$ function $h:\mathbb{R}\to \mathbb{R}$, introduce for $i\in \{1,2\},$
\begin{align*}
\gamma _{i}(x,f,h)=\mathbf{1}_{\{\nabla f(x)\in Q_{i}\}}\left[
\partial _{i}f(x)^{2}h'(f(x))+\partial^2_{ii}f(x)h(f(x))
\right], x\in W,
\end{align*} $\gamma (x,f,h)=\gamma _{1}(x,f,h)+\gamma _{2}(x,f,h)$, and $I_{f}(h)=\int_{W}\gamma (x,f,h)$.   Along the text, we might ask additional properties from the test function $h$, such as that to be twice continuously differentiable, or have compact support.

\begin{theorem} 
\label{thm:main}Let $h:\mathbb{R}\to \mathbb{R}$ be a $\mathcal{C}^{1}$   function with compact support.
Let $f:W\subset \mathbb{R}^{2}\to \mathbb{R}$ be a $\mathcal{C}^{2}$  function which is Morse above $\min(\supp(h))$, and  such that $\{f\geqslant \min(\supp(h)) \}$ is compact and contained in $W$'s interior.  Then {  $\chi _{f}(h)$ is well defined and } 
\begin{align}
\label{eq:main-result}
\chi_{f}(h)   = I_{f}(h). \end{align} 
\end{theorem}
The proof is postponed to the Appendix, in  Section \ref{sec:main-proof}. It relies on the application of Theorem 1.1 from \cite{LacEC1}, that states that for $u\notin V(f)$,  the \EC~can be expressed as the limit of some quantity $\delta _{u,\varepsilon }\in \mathbb{R}$ that is explicit in \cite{LacEC1},
\begin{align*}
\chi (\{f\geqslant u\})=\lim_{\varepsilon \to 0}\delta _{\varepsilon ,u}(f).
\end{align*}The quantity $ | \delta _{\varepsilon ,u } | $ cannot be bounded uniformly over $\varepsilon $ in some quantity integrable in $u$, because of its behavior around the critical values of $f$. This difficulty, which prevents us from directly  switching $\lim_{\varepsilon \to 0}$ and $\int_{\mathbb{R}}$,
compels us to apply this formula to a random perturbation of $f$, noted $f_{\eta },\eta >0$. Then the results of \cite{LacEC2}, concerning random fields, can be applied to each $f_{\eta }$, using the density of this fields' marginals, and then use Lebesgue's theorem in the limit $\eta \to 0.$ This randomization of the problem allows us to avoid dealing with the quantity $\chi (\{f\geqslant u\})$ when $u$ is close to a critical value of $f$, but it also raises doubts as to the optimality of such a proof.\\

For $f$ like in Theorem  \ref{eq:main-result} and $u\notin V(f)$, \eqref{eq:intro-classical-formula} yields that $\chi (\{f\geqslant u\})$ is constant on a neighbourhood of $u$.  For $\varepsilon $ sufficiently small and $\delta_{\varepsilon } :\mathbb{R}\to \mathbb{R}$ of class $\mathcal{C}^{1}$ with support in $[-\varepsilon ,\varepsilon ]$ such that $\int_{-\varepsilon }^{\varepsilon }h(v)dv=1$,
\begin{align*}
\chi ({f\geqslant u})=\lim_{\varepsilon \to 0}\chi _{f}(\delta _{\varepsilon })=\lim_{\varepsilon \to 0}I_{f}(\delta _{\varepsilon }).
\end{align*}
This formula can be seen as a $2$-dimensional analogue of the celebrated Kac-Rice formula, obtained by taking $h=\delta _{\varepsilon }$ in \eqref{eq:kacrice-intro}.

In the context of a random field $f$, \eqref{eq:main-result}  can be passed on to the expectation. Let $(\Omega ,\mathcal{A},\P)$ be a complete probability space. 
Let $W\subset \mathbb{R}^{2}$ open. We call here $\mathcal{C}^{2}$ random field a collection of random variables $\{f(x),x\in W\}$ such that for any countable subset $I\subset W, \{f(x),x\in I\}$ has a unique $\mathcal{C}^{2}$ extension  on $W$, still denoted $f$, and the finite dimensional distributions of $f$ do not depend on the choice of $I$. See \cite{AdlTay07} for more on the formalism of random fields. Say that $f$ is Morse \emph{above} some value $\th\in \mathbb{R}$ if with probability $1$, $\{f(x),x\in W\}$, is Morse above $\th.$

\begin{corollary}
\label{cor:main-random}Let $W\subset \mathbb{R}^{2}$ open,  
 $f:W\to \mathbb{R}$   a random $\mathcal{C}^{2}$ function, and $\th  _{0}\in \mathbb{R}$. Assume that   $\{f\geqslant \th  _{0}\}$ is a.s. compact and that $f$ is a.s. Morse above $\th_{0}$. Then, if 
\begin{align*}
\int_{\mathbb{R}^{2}}\E\left[
\mathbf{1}_{\{f(x)\geqslant \th  _{0}\}}\left[
\|\nabla f(x)\|^{2}+  | \partial _{11}^2f(x) | + | \partial _{22}f(x) | 
\right]
\right]dx<\infty ,
\end{align*}
for every $\mathcal{C}^{1}$ function $h:\mathbb{R}\to \mathbb{R}$ with support in $[\th  _{0},\infty )$, $\E\left[
 | \chi _{f}(h) | 
\right]<\infty$ 
and 
\begin{align*}
\E \chi _{f}(h)=\E I_{f}(h)=\int_{\mathbb{R}^{2}}\E \gamma (x,f,h)dx.
\end{align*}  
\end{corollary}
\begin{proof} We have a.s., according to Th. \ref{thm:main}, 
\begin{align*}
\chi _{f}(h)=I_{f}(h),
\end{align*}and it is clear that  the assumption yields 
\begin{align*}
\int_{W}\E\left[
 |\gamma  (x,f,h) | \right]dx<\infty 
,
\end{align*}whence Lebesgue's theorem yields
\begin{align*}
\E | \chi _{f}(h) | =\E | I_{f}(h) | <\E\left[
\int_{} | \gamma (x,f,h) | dx
\right]<\infty ,
\end{align*}and the conclusion follows by switching integral and expectation.
\end{proof}

\begin{remark}
\label{rmorse}
It is in general, difficult to verify that a random $\mathcal{C}^{2}$ field $f$ has a.s. Morse sample paths, see for instance \cite[Section 5]{AST} for an abstract result. In the case where $f$ is  a centred Gaussian field, some explicit  necessary conditions exist, see Corollary 11.3.2 in \cite{AdlTay07}: Assume that the vector formed by the partial derivatives $(\partial _{i}f(x),\partial _{ij}f(x))_{1\leqslant i\leqslant j\leqslant 2},x\in \mathbb{R}^{2}$ is non-degenerate, and that the covariance function, for $1\leqslant i,j\leqslant 2$,
\begin{align*}
\Sigma _{i,j}(x,y)=\E \left[
\partial ^{2}_{i,j}f(x)\partial ^{2}_{i,j}f(y)
\right],x,y\in W,
\end{align*}satisfies for some $C_{W}>0$
\begin{align*}
\left|
\Sigma _{i,j}f(x,x)+\Sigma _{i,j}f(y,y)-2\Sigma _{i,j}f(x,y)
\right|\leqslant C_{W} | \ln(\|x-y\|) | ^{-1+\alpha },x,y\in W,
\end{align*}for some $\alpha >0$. Then the sample paths of $f$ are a.s. Morse over $W$.
\end{remark}

\begin{remark}
\label{rmk-isotropy}
If $f(x)$ is locally isotropic in some point $x\in W$, i.e. if the law of $\nabla f(x)$ is invariant under rotations, conditionally to $f(x)$, then for any bounded measurable function $h$
 the first integrand of $I_{f}(h)$ simplifies to
  \begin{align*}
\E \left[
h(f(x))\mathbf{1}_{\{\nabla f(x)\in Q_{i}\}}\partial _{i}f(x)^{2}
\right]=\frac{\pi -2}{16\pi } \E\left[
 h(f(x)) \|\nabla f(x)\|^{2}
\right],i=1,2.
\end{align*}To show it,  note in $\mathbb{R}^{2}$ $\S^{1}$ the unit circle, $(\cdot )$ the canonical scalar product, and $\mathcal{H}^{1}$ the $1$-dimensional Hausdorff measure in $\mathbb{R}^{2}$.
Conditionnally to $f(x)$, one can decompose the law of $\nabla f(x)$ in the couple of independent variables $\left(\|\nabla f(x)\|,\frac{\nabla f(x)}{\|\nabla f(x)\|}\right)$. The law of  $\frac{\nabla f(x)}{\|\nabla f(x)\|}$ is furthermore uniform in $\S^{1}$. Therefore, for $i=1$,  noting that $\partial _{1}f(x)=\|\nabla f(x)\| \left(
\frac{\nabla f(x)}{\|\nabla f(x)\|}\cdot \u_{1}
\right)$,
\begin{align*}\E \left[
h(f(x))\mathbf{1}_{\{\nabla f(x)\in Q_{1}\}}\partial _{1}f(x)^{2}
\right]&={  \frac{1}{2\pi }}
\int_{\S^{1}}\mathbf{1}_{\{u\in Q_{1}\}}\E\left[
( u\cdot \u_{1} )^{2} h(f(x))\|\nabla f(x)\|^{2}
\right]d\mathcal{H}^{1}(u )\\
&={ \frac{1}{2\pi }}\int_{-3\pi/4 }^{-\pi /2}  \cos(\theta )^{2}\E\left[
 h(f(x))\|\nabla f(x)\|^{2}
\right]d\theta \\
&=\frac{\pi -2}{16\pi }\E \left[
h(f(x))\|\nabla f(x)\|^{2}
\right],
\end{align*}and the  same computation holds for $i=2$ (with $\int_{-3\pi /4}^{-\pi /2}\cos(\theta )^{2}d\theta $  replaced by $\int_{-\pi }^{-3\pi /4}\sin(\theta )^{2}d\theta $, also equal to $(\pi -2)/8$).
With similar arguments, 
this term is also easy to compute if the function is  radial and deterministic.  See Theorem \ref{thm:stat-SN} for an illustration  in the framework of shot noise processes.
\end{remark}

 \begin{remark} 
 \label{rmk:r-theta}
 Call $r_{\theta }$ the clockwise rotation with angle $\theta \in [0,2\pi ]$ in the plane. For a function $f$ defined on the plane, put $f^{\theta }(x)=f(r_{\theta }(x))$. The invariance of the \EC~under rotation yields that for all $\theta ,\chi _{f^{\theta }}=\chi _{f}$. In particular, we have for any test function $h$, averaging over $r_{0},r_{\pi /2},r_{\pi },r_{3\pi /2},$
\begin{align*}
\chi _{f}(h)=\frac{-1}{4}\sum_{i=1}^{2}\int_{\mathbb{R}^{2}}\mathbf{1}_{\{ | \partial _{i'}f(x) | > | \partial _{i}f(x) | \}}\left[
\partial _{i}f(x)^{2}h'(f(x))+\partial^2_{ii}f(x)h(f(x))
\right]dx.
\end{align*}where 
\begin{align*}
i'=\begin{cases}1$ if $i=2\\
2$ if $i=1.
\end{cases}
\end{align*}
 We also have the formula 
\begin{align*}
\chi _{f}(h)=\frac{1}{2\pi }\int_{0}^{2\pi }I_{f^{\theta }}(h)d\theta.
\end{align*}
  \end{remark}
 
\subsection{Extension to non-Morse functions}
\label{sec:non-morse}

In the work \cite{LacEC2}, the validity of the result about the \EC~of excursions only requires $\C^{1,1}$ regularity, i.e. continuous differentiability with Lipschitz gradient. In contrast, Theorem \ref{thm:main} requires $\mathcal{C}^{2}$ regularity and Morse behaviour around the critical points. Still we believe that the conclusion could be valid under $\mathcal{C}^{1,1}$ regularity (under such assumptions, the second order partial derivatives are well defined a.e.).

 To support and motivate this claim, we show here that it holds  if the function $f$ is radial.
 
  \begin{theorem}
\label{ex:radial}
 Assume $f(x)=\psi (\|x^{2}\|),x\in \mathbb{R}^{2},$ for some  $\C^{1,1}$  function $\psi:\mathbb{R}_{+}\to \mathbb{R}_{+} $ that vanishes at $\infty $.
Then for a.a. $u>0$, $\chi (\{f\geqslant u\})$ is well-defined and bounded by $1$, and for
  $h:\mathbb{R}_{+}\to \mathbb{R}_{+}$ a $\mathcal{C}^{1}$ function with compact support in $(0,\infty )$, we have
\begin{align*}
\chi _{f}(h)=I_{f}(h)=\int_{0}^{\psi (0) }h(u)du.
\end{align*}
\end{theorem} 

 It is not clear what a general result should be in dimension $2$, and how to prove it.  By modifying the proof in Section \ref{sec:main-proof} (after the proof of Lemma  \ref{lm:cvg-critical-points}),
 an abstract condition replacing the Morse assumption could be that there is some $M>0,n_{0}>0$ such that for Lebesgue-almost all $A,B_{1},B_{2}\in [-M,M]$, the function  $f_{\eta  }:x\mapsto f(x)+\eta (A+B_{1}x_{1}+B_{2}x_{2})$ is Morse and has less  than $n_{0}$ critical points.

\begin{proof}[Proof of Theorem \ref{ex:radial}]
 
The $\mathcal{C}^{1,1}$ regularity amounts to the fact that $\psi '$ has a.e. a derivative, noted $\psi ''$, or equivalently that is is absolutely continuous, i.e. 
\begin{align*}
\psi '(b)-\psi '(a)=\int_{a}^{b}\psi ''(x)dx
\end{align*}for all $a,b\in \mathbb{R}_{+}$.

Recall that a point $x\in \mathbb{R}_{+}\setminus M(\psi )$ is an \emph{up-crossing} at level $u>0$ if $\psi (x)=u$ and $\psi '(x)>0$, and a \emph{down-crossing} if $\psi (x)=u$ and $\psi '(x)<0$. Call $N_{u}^{+}(\psi )$ the total number of up-crossings at level $u$, and $N_{u}^{-}(\psi )$ the number of down-crossings. From \eqref{eq:coarea-intro}, we have  
\begin{align}
\label{eq:1D-rice}
\int_{\mathbb{R}_{+}}N_{u}^{+}(\psi )du=\int_{\mathbb{R}_{+}}N_{u}^{-}(\psi )du=\int_{\mathbb{R}_{+}}   | \psi '(x) | dx .
\end{align} Since   $V(\psi )$ is negligible, it yields that the number of connected components of $\{\psi \geqslant u\}$ is finite for a.a. $u$.

  For  $u>0$ not in $V(\psi )$, $\{f\geqslant u\}$ is a union of $N_{u}^{-}(\psi )$ concentric rings, with the central ring being a disc if $0\in \{\psi \geqslant u\}$, and a proper ring otherwise. Furthermore, as a union of concentric rings, the \EC~of $\{f\geqslant u\}$ is $1$ if $0\in \{f\geqslant u\}$, and $0$ otherwise. It yields for a.a. $u>0$, 
\begin{align}
\label{eq:radial-chi-h}
\chi (\{f\geqslant u\})=\mathbf{1}_{\{f(0)\geqslant u\}}=\mathbf{1}_{\{\psi (0)\geqslant u\}},
\end{align}whence $\chi _{f}(h)$ is clearly well-defined. 

Let us now compute $I_{f}(h)$.
We have for a.a. $x\in \mathbb{R}^{2},i=1,2,$
\begin{align*}
\partial _{i}f (x)=2x_{i}\psi '(x^{2}),\hspace{2cm}\partial^2_{ii}f (x)=2\psi '(x^{2})+4x_{i}^{2}\psi ''(x^{2}).
\end{align*}
If $x$ has polar coordinates $(r,\theta)\in \mathbb{R}_{+}\times  [-\pi ,\pi ] $, then for $i=1,2,$ $\nabla f(x)\in Q_{i}$ amounts to $\theta \in I_{i}$ with $I_{1}=[-3\pi/4,-\pi /2]$, and $I_{2}=[-\pi  ,-3\pi /4]$.
We therefore have 
\begin{align*}
I_{f}(h)=& \int_{I_{1} }\int_{0}^{\infty } \left[
h'(\psi (r^{2}))(2r\cos (\theta )\psi '(x^{2}))^{2}+h(\psi (r^{2}))\left(
2\psi '(r^{2})+4(r\cos (\theta ))^{2}\psi ''(r^{2})
\right)
\right]rdrd\theta \\
&+\int_{I_{2} }\int_{0}^{\infty } \left[
h'(\psi (r^{2}))(2r\sin (\theta )\psi '(x^{2}))^{2}+h(\psi (r^{2}))\left(
2\psi '(r^{2})+4(r\sin (\theta ))\psi ''(r^{2})
\right)
\right]rdrd\theta \\
=& \left(
\int_{I_{1} } 
\cos(\theta )^{2} d\theta
\right)
 \int_{0}^{\infty } \left[
h'(\psi (r^{2}))4r^{2} \psi '(x^{2})^{2}+h(\psi (r^{2})) 
 4 r ^{2}\psi'' (r^{2})
\right]rdrd\theta \\
&+ \left(
\int_{I_{2} } 
 \sin(\theta )^{2}
 d\theta
\right)
 \int_{0}^{\infty } \left[
h'(\psi (r^{2}))4r^{2} \psi '(x^{2})^{2}+h(\psi (r^{2})) 
 4 r ^{2}\psi ''(r^{2})
\right]rdrd\theta \\
&+2\cdot \frac{\pi }{4}\int_{0}^{\infty }2h(\psi (r^{2}))\psi '(r^{2})rdr\\
=&\left(
2\int_{0}^{\pi /4}\sin(\theta )^{2}d\theta 
\right)\left(
2I+2J
\right)+\frac{\pi }{2}K
\end{align*}
where, with the change of variables $u=r^{2}$,  
\begin{align*}
I&=\int_{0}^{\infty }h'(\psi (u))\psi '(u)^{2}udu\\
J&=\int_{0}^{\infty }h(\psi (u))\psi ''(u)udu\\
K&=\int_{0}^{\infty }h(\psi (u))\psi '(u)du.\end{align*}
An integration by parts gives 
\begin{align*}
J &=[\psi '(u)h(\psi (u))u]_{0}^{\infty }-\int_{0}^{\infty }\psi '(u)[h(\psi (u))+u\psi '(u)h'(\psi (u))]du\\
&=-\int_{0}^{\infty }h(\psi (u))\psi '(u)du-\int_{0}^{\infty }u\psi '(u)^{2}h'(\psi (u))du=-K-I.\\
\end{align*}
Then, 
\begin{align*}
\int_{0}^{\pi /4}\sin^{2}(\theta )d\theta =\frac{\pi -2}{8}.
\end{align*}
 We finally have 
\begin{align}
\label{eq:radial}
I_{f}(h)&=\frac{\pi -2}{2}(-K)+K\frac{\pi }{2}=K.
\end{align} 
To conclude, remark that if $a<b\in \mathbb{R}_{+}$ are two consecutive zeros of $\psi '$ such that $\psi '>0$ on $(a,b)$,
\begin{align*}
\int_{a}^{b}h(\psi (u))\psi '(u)du=\int_{\psi (a)}^{\psi (b)}h(v)dv,
\end{align*}and if $\psi '<0$ on $(a,b)$, $\psi (b)<\psi (a)$ and
\begin{align*}
\int_{a}^{b}h(\psi (u))\psi '(u)du=-\int_{\psi (b)}^{\psi (a)}h(v)dv.
\end{align*}
Call $\mathcal{I}^{+}$ the set of maximal open intervals of $\mathbb{R}_{+}$ where $\psi '>0$, and $\mathcal{I}^{-}$ the set of maximal open intervals of $\mathbb{R}_{+}$ where $\psi '<0$. For $u\in \mathbb{R}_{+}\setminus V(\psi )$, 
$N_{v}^{+}(\psi )$ is the number of intervals $I$ of $\mathcal{I}^{+}$ such that $u\in \psi (I)$, and $N_{u}^{-}(\psi )$ the number of $I\in \mathcal{I}^{-}$ such that $u \in \psi (I)$.
 
Decomposing \eqref{eq:radial} as a sum over all open maximal intervals where $\psi '\neq 0$ yields  
\begin{align*}
I_{f}(h)=\int_{\mathbb{R}}h(u)(N_{u}^{+}(\psi )-N_{u}^{-}(\psi ))du.
\end{align*}\eqref{eq:1D-rice} yields that for a.a. level $u>0$, the number of down crossings and up-crossings are finite.
Let $u>0$ that is not a critical value of $\psi $. If $u>\psi (0)$, since $\psi (u)\to 0$ as $u\to \infty ,$ every  upcrossing at level $v$  can be uniquely associated with a downcrossing, namely the smallest point $u'>u$ where $\psi (u')=u$ and $\psi '(u')<0$. It follows that $N_{u}^+(\psi )=N_{u}^{-}(\psi )$. For $0<u<\psi (0)$, the first downcrossing at $u$ cannot be matched with any upcrossing, therefore $N_{u}^{-}(\psi )=N_{u}^{+}(\psi )+1$.
We indeed have shown that   
\begin{align*}
I_{f}(h)=\int_{\mathbb{R}}\mathbf{1}_{\{u>\psi (0)\}}h(u)du=\chi _{f}(h)
\end{align*}in virtue of \eqref{eq:radial-chi-h}.

\end{proof}

\subsection{Continuity in $f$}
\newcommand{\e}{\mathbf{e}}For $f$ a $\mathcal{C}^{2}$ function defined on some measurable subset $W\subset \mathbb{R}^{2}$,  we have the following result for the continuity of $I_{f}$ in  $f$. For $\u\in \S^{1}$, note $\partial _{\u}f$ the partial derivative of $f$ in direction $\u$. Introduce  $\e_{1}=2^{-1/2}(\u_{1}+\u_{2}),\e_{2}=2^{-1/2}(\u_{1}-\u_{2})$.
 Note $\delta _{x}(f,g)=\max_{\u\in \{\u_{1},\u_{2},\e_{1},\e_{2}\}}\left(
\mathbf{1}_{\{ | \partial _{\u}f(x) | \leqslant  | \partial _{\u}g(x) | \}}
\right).$We also note, for some function $g:A\subset \mathbb{R}^{m}\to \mathbb{R}$ of class $\C^{k},k\geqslant 0,$ and $0\leqslant p\leqslant k$, 
\begin{align*}
\|g^{(p)}\|=\sup_{x\in A}\max_{(i_{1},\dots ,i_{p})\in  \{1,\dots ,m\}^{p}}\left|
\frac{\partial ^{p}g(x)}{\partial _{i_{1}}\dots \partial _{i_{p}}}
\right|,
\end{align*}and   $N_{p}(g)=\max_{0\leqslant i\leqslant p}\|g^{(i)}\|.$

\begin{lemma}
\label{ref:lm-continuity}
Let $h$ be a $\mathcal{C}^{2}$ real function with compact support in $\mathbb{R}$.
Given two $\C^{2}$ functions  $f,g$ we have the bound, for $x\in \mathbb{R}^{2}$ in both their domains of definition, and $i=1,2,$ 
\begin{align}
\label{eq:general-continuity-gamma}
\notag\left|
\gamma_{i} (x,f+g,h)-\gamma (x,f,h)
\right| \leqslant &6 
N_{2}(h)\max \left(
\partial _{i}f(x)^{2}, | \partial^2_{ii}f(x) | ,2 | \partial _{i}f(x) | + | \partial _{i}g(x)
\right)\\
&\times\max  \left(
\delta _{x}(f,g),  | \partial _{i}g(x) | ,| g(x) | , | \partial^2_{ii}g(x) | 
\right).\end{align}
\end{lemma} 

\begin{proof} 
\begin{align*}
\gamma_{i} (x,f,h)-\gamma_{i} (x,f+g,h)=& 
\partial _{i}f(x)^{2}h'(f(x))\left(
 \mathbf{1}_{\{\nabla f(x)\in Q_{i}\}}-\mathbf{1}_{\{\nabla (f+g)(x)\in Q_{i}\}}
\right)\\
&+
\partial _{i}f(x)^{2}\mathbf{1}_{\{\nabla (f+g)(x)\in Q_{i}\}}\left[
h'(f(x))-h'((f+g)(x))
\right]\\
&+\left(
\partial _{i}f(x)^{2}-\partial _{i}(f+g)(x)^{2}
\right)h'((f+g)(x))\mathbf{1}_{\{\nabla (f+g)(x)\in Q_{i}\}}\\
&+\partial^2_{ii}f(x)h(f(x))\left(
\mathbf{1}_{\{\nabla f(x)\in Q_{i}\}}-\mathbf{1}_{\{\nabla (f+g)(x)\in Q_{i}\}}
\right)\\
&+
\partial^2_{ii}f(x)\mathbf{1}_{\{\nabla (f+g)(x)\in Q_{i}\}}\left[
h(f(x))-h((f+g)(x))
\right]\\
&+\left(
\partial^2_{ii}f(x)-\partial^2_{ii}(f+g)(x)
\right)h((f+g)(x))\mathbf{1}_{\{\nabla (f+g)(x)\in Q_{i}\}}
.
\end{align*} Note that for any two vectors $u,v\in \mathbb{R}^{2}$, 
\begin{align*}
\left|
\mathbf{1}_{\{u\in Q_{1}\}}-\mathbf{1}_{\{u+v\in Q_{1}\}}
\right|=& \left|
\mathbf{1}_{\{u_{2}<u_{1}<0\}}-\mathbf{1}_{\{u_{2}+v_{2}<u_{1}+v_{1}<0\}}
\right|\\
\leqslant &  \left(
\mathbf{1}_{\{ | u_{1} | < | v_{1} | \}}+\mathbf{1}_{\{ | u_{1}-u_{2} | < | v_{1}-v_{2} | \}}
\right) .
\end{align*} 
Using also $\left|
h'(f(x))-h'((f+g)(x))
\right|\leqslant \|h^{(2)}\| | g(x) | $ and $ | \partial _{i}f(x)^{2}-\partial _{i}(f+g)(x)^{2} | \leqslant  | \partial _{i}g(x) | 
\left(2
\left|
\partial _{i}f(x)
\right|+\left|
\partial _{i}g(x)
\right|
\right) $ the previous expression is bounded by
\begin{align*}
6 &
N_{2}(h)
\max\left(
\partial _{i}f(x)^{2}, | \partial^2_{ii}f(x) | ,2 | \partial _{i}f(x) | + | \partial _{i}g(x)
\right)\max\left(
\delta _{x}(f,g),  | \partial _{i}g(x)|, |  g(x) | , | \partial^2_{ii}g(x) | 
\right).
\end{align*}
\end{proof}
The idea of this inequality is that $I_{f+g}(h)$ and $I_{f}(h)$ are close when $g$ is small with respect to the norm $N_{2}(\cdot )$, and the gradient of $f$ does not lie too close from the boundaries of the quarter planes $Q_{i}$. When the fields are random, the most delicate quantity to deal with is the probability $\P( | \partial _{\u}f(x) | \leqslant  | \partial _{\u}g(x) | ),\u\in \S^{1}$, which requires a fine control of the law of $\partial _{\u}f(x)$ near $0$. In many cases, such as when $f$ is a shot noise field (see Section \ref{sec:SN}), the existence and boundedness of the density around $0$ is problematic, and requires for instance that the grain functions of the shot noise process have $\mathbb{R}^{2}$ as their support; see \cite{AST,BieDes}.

In an asymptotic study of the \EP, one will typically need to control the variation of higher order moments under the variation of the function $f$. We derive the following result, that is used in the proof of Theorem \ref{thm:stat-SN} for an asymptotic formula of the \EP~for stationary shot noise processes.

 We extend the notation of the Euler primitive to $I_{f}(h,W'):=\int_{W'}\gamma (x,f,h)dx$, when either $W'\subset W$ is compact, or $(\partial _{i}f)^{2}$ and $ | \partial^2_{ii}f | $ are integrable over $W'$ for $i=1,2$.

\begin{corollary}
\label{ref:lm-continuity-random}
\label{cor:continuity}
 Let $q\geqslant 1$, and $h:\mathbb{R}\to \mathbb{R}$ a $\mathcal{C}^{2}$ function with compact support.
 Let  $f,g$ be two random $\mathcal{C}^{2}$ fields over some measurable set $W\subset \mathbb{R}^{2}$, such that $I_{f}(h,W)$ and $I_{g}(h,W)$ have a finite $q$-th moment. Then there is $C_{q}>0$ not depending on $f,g,h,$ or $W$, such that
 \begin{align*}
\E& \left|
I_{f}(h,W)-I_{f+g}(h,W)
\right|^{q}\\
 &\leqslant  C_{q} 
N_2(h)^{q} \max_{i=1,2}\Bigg(\int_{W}\Bigg[\max\left(
\E\partial _{i}f(x)^{4q},\E | \partial^2_{ii}f(x) |^{2q} ,2\E  \partial _{i}f(x)^{2q}  +\E   \partial _{i}g(x)^{2q}
\right) \\
& \times\max\left(
 \E | g(x) |^{2q} ,\E | \partial _{i}g(x) |^{2q} ,\E | \partial^2_{ii}g(x) |^{2q} 
, \delta _{x}(f,g) \right)\Bigg]^{1/2q}dx\Bigg)^{q}
\end{align*}
\end{corollary}

\begin{proof}

\begin{align*}
\E & \left|
I_{f}(h,W)-I_{f+g}(h,W)
\right|^{q} \\
  \leqslant &\sum_{1\leqslant i_{1},\dots ,i_{q} \leqslant 2} \int_{W^{q}} \E( \gamma _{i_{1}}(x_{1},f,h)- \gamma _{i_{1}}(x_{1},f+g,h)
)\\
&\hspace{5cm}\dots ( \gamma _{i_{q}}(x_{q},f,h)- \gamma _{i_{q}}(x_{q},f+g,h))dx_{1}\dots dx_{q}\\
 \leqslant &\sum_{1\leqslant i_{1},\dots ,i_{q} \leqslant 2} \int_{W^{q}} \left[
\E( \gamma _{i_{1}}(x_{1},f,h)- \gamma _{i_{1}}(x_{1},f+g,h)
)^{q}
\right]^{1/q}\\
&\hspace{5cm}\dots \left[
\E( \gamma _{i_{q}}(x_{q},f,h)- \gamma _{i_{q}}(x_{q},f+g,h))^{q}
\right]^{1/q}dx_{1}\dots dx_{q}\\
\leqslant &6 ^{q}
N_{2}(h)^{q}\sum_{1\leqslant i_{1},\dots ,i_{q}\leqslant 2}\prod_{k=1}^{q}\Bigg(\int_{W } \Big( \E\Big[\max \left(
\partial _{i_{k}}f(x )^{2}, | \partial^2_{i_{k}i_{k}}f(x ) | ,2 | \partial _{i_{k}}f(x ) | + | \partial _{i_{k}}g(x )
\right)\\
&\hspace{4cm}\times\max \left(
\delta _{x }(f,g),  | \partial _{i_{k}}g(x ) | ,| g(x ) | , | \partial^2_{i_{k}i_{k}}g(x ) | 
\right)\Big]^{q}\Big)^{1/q}dx\Bigg)^{q}\\
\leqslant &C_{q}
N_{2}(h)^{q}\max_{i=1,2}\Bigg( \int_{W }  \Big[\E\max \left(
\partial _{i}f(x)^{4q}, | \partial^2_{ii}f(x) | ^{2q},  | \partial _{i}f(x) |^{2q} + | \partial _{i}g(x) | ^{2q}
\right)\Big]^{1/2q} \\
&\times\left[
 \E\max \left(
\delta _{x}(f,g),  | \partial _{i}g(x) |^{4q} ,| g(x) |^{2q} , | \partial^2_{ii}g(x) | ^{2q}
\right)
\right]^{1/2q}dx\Bigg)^{q}.
\end{align*}

\end{proof}

\section{Expectation of the Euler characteristic of a Shot Noise process}
\label{sec:SN}

Most results giving the expected \EC~of a random field require the marginals of the fields to satisfy a density hypothesis (Theorem 11.2.1 in \cite{AdlTay07}, Condition 5.2 in \cite{AST}, or Proposition 4 in \cite{BieDes12}). This kind of assumptions are hard to resolve outside the Gaussian realm. We show in this section how to deal with the \EP~with shot noise fields, without density assumptions. Let us first introduce  this family of functions.

Shot noise processes, also called moving averages,  sparse convolution models, or many other names, are used for modelisation in many fields: Telecommunications, texture synthesis, neurobiology. We introduce here a non-parametric rather abstract family of shot noise processes, that can be specified to many models encountered in the literature.    
\newcommand{\G}{\mathcal{G}}
 Let $\mathcal{G}$ be the set of  non-negative continuous real functions $g$ on $\mathbb{R}^{2}$  such that for $\th>0, \{g\geqslant u\}$ is compact and $g$ is Morse above $u$. Such functions necessarily vanish at $\infty $ and can possibly have degenerate critical points corresponding to the critical value  $0$. See Example \ref{ex:g0} for some admissible functions.

 Let  $\mathcal{B}$  be the Borel $\sigma $-algebra of  uniform convergence on each compact set, and $\mu $ be a probability measure on $(\G,\mathcal{B})$ such that 
\begin{align}
\label{eq:basic-int-mu}
\int_{\mathbb{R}^{2}\times \G} | g(x) | dx\mu (dg)<\infty .
\end{align}  
 For $W\subset \mathbb{R}^{2}$ compact, let $\N(W)$ be the space of finite sets on $W$, and $\N(W\times \G)$ the space of finite sets on $W\times \G$. Let $\tilde \eta^{W} \in \N(W\times \G)$ be a Poisson process on $W \times \mathcal{G}$ with intensity measure $\ell\mathbf{1}_{\{W\}}\otimes \mu $ (see \cite{LastChapter} for a proper introduction to Poisson measures on arbitrary spaces).

Note $\eta^{W}\in \mathcal{N}(W) $ the Poisson point process that consists of the spatial projections of points of $\tilde \eta  ^{W}$. If $W$ is implicit from the context, note $\tilde \eta =\tilde \eta ^{W}$ and $\eta =\eta ^{W}$ for simplicity. For $x\in \eta ,y\in \mathbb{R}^{2}$, note $g_{x}(y)=g(y-x)$ where $g$ is the a.s. unique function such that $(x,g)\in \tilde \eta $.   In virtue of \eqref{eq:basic-int-mu}, we can define the shot noise process with germ process $\eta  $ and kernel model $\mu $ as 
\begin{align}
\label{eq:def-SN}
f(x)=\sum_{y\in   \eta  }g_{y}(x),x\in \mathbb{R}^{2},
\end{align}
see \cite{HeiSch} for details.

Remark that the field $f$ is a.s. twice continuously differentiable over the whole plane, even though the  points of $\eta $  only fall in $W$. The hypotheses on $\G$ imply that  with probability $1$, $f$ is non-negative
and vanishes at $\infty $.
It is not easy to give good conditions on $\mu $ that ensure that $f$ is indeed a Morse function (over the levels $\th  >0$), see \cite[Section 5]{AST} for a discussion on this topic.

The distribution of shot noise fields marginals are more easily expressed via the characteristic function. Likewise, it is more tractable to compute the Fourier transform of the  \EC. Noting $h^{(t)}(\th  )=\exp(\imath t\th  ),t,\th\in \mathbb{R}$, define
\begin{align*}
\widehat{\chi _{f }}(t )=\int_{\mathbb{R}}h^{(t)}(\th  )\chi (\{f \geqslant \th\}  )d\th  ,t\in \mathbb{R}.
\end{align*}
At this stage, this expression is not properly defined since $h^{(t)}$ does not have a compact support in $(0,\infty )$.
Applying Corollary \ref{cor:main-random} requires a little bit of care, especially to deal with the discontinuity of $\th  \mapsto \chi (\{f\geqslant \th  \})$ at $0$. 
  For that we need additional structural assumptions on the typical grain to control the topological complexity of excursion sets $\{g\geqslant u\}$ as $u\to 0.$
\begin{assumption}
\label{ass:g-structural}
We assume that for $\mu $-almost every 
$g\in \mathcal{G}$, there is $\th  _{g}>0$ such that the excursion sets $\{g\geqslant \th  \},0<\th  <\th  _{g}$, are convex. Assume furthermore that a.s.
\begin{align}
\label{eq:integrability-g}
\int_{\mathbb{R}^{2}}\sum_{i=1}^{2}[\partial _{i}g(x)^{2}+ | \partial^2_{ii}g(x)]dx<\infty ,
\end{align}and that there is $\alpha_{g} \geqslant 1,$ and a random variable $C_{g} >0$ such that a.s. 
\begin{align*}
\|\nabla g(x)\|\leqslant  C_{g}| g(x) | ^{ \alpha_{g} /2}, x\in  {\mathbb{R}^{2}},
\end{align*} and 
\begin{align*}
 \int_{\mathbb{R}^{2}} | g(x) | ^{\alpha_{g} -1}dx <\infty .
\end{align*}

\end{assumption}

 A typical example  consists of taking $\nu =\delta _{g_{0}}$ for some fixed $g_{0}\in \G$ satisfying the assumptions above.  In this case, for $x\in \mathbb{R}^{2},$ the random variable $f(x) $ might not  have a bounded density, and traditional results cannot be applied to even prove that the \EC~has a finite expectation. We have the following result for the Euler  primitive:
\begin{theorem}  
\label{thm:EC-SN}Let the previous notation prevail.
Assume that $W$ is compact,   that $\mu $ satisfies Assumption \ref{ass:g-structural} and that $f$ as defined in \eqref{eq:def-SN} is a.s. Morse above $\th,\th>0$. Then for $t\in \mathbb{R}$, a.s,
\begin{align}
\notag
\widehat{\chi _{f}}(t)=&I_{f}(h^{(t)})\\
\label{eq:EC-SN-bounded}=&-\sum_{i=1}^{2}\int_{\mathbb{R}^{2}}\E\left[
\exp(\imath tf(x))\mathbf{1}_{\{\nabla f(x)\in Q_{i}\}}\left[
it\partial _{i}f(x)^{2}+\partial^2_{ii}f(x)
\right]
\right]dx.
\end{align}

\end{theorem}

  See Section \ref{sec:proof-SN} in the Appendix for the proof. 
Let us apply this result to the approximation of a stationary isotropic shot noise process.
Introduce $W_{n}=B(0,\sqrt{n})$, and note $\tilde\eta _{n}:=\tilde \eta ^{W_{n}},\eta_{n}=\eta ^{W_{n}}$. Let $f_{n}$ be the corresponding shot-noise field, as defined on $\mathbb{R}^{2}$ by \eqref{eq:def-SN}. Assumptions \eqref{eq:basic-int-mu} and \eqref{eq:integrability-g} ensure that  $f (x)=\lim_{n}f_{n}(x)$ is a.s. well defined for every $x\in \mathbb{R}^{2}$, and has finite first and second-order derivatives in every $x\in \mathbb{R}^{2}$ (see \cite{HeiSch}).

Introduce the characteristic function, for $ t\in \mathbb{R},s=(s_{1},s_{2})\in \mathbb{R}^{2},v\in \mathbb{R},$
\begin{align}
\notag\psi _{1}(t,s,v)=&\E\exp\left(
\imath \left[
tf(0)+(s\cdot \nabla f(0))+v\partial^2_{11}f(0)
\right]
\right)\\
\label{eq:FT-SN}=&\exp\left(
\int_{\mathbb{R}^{2}\times \G}\left(
\exp\left[
\imath  \left(
tg(x)+(s\cdot \nabla g(x))+v\partial^2_{11}g(x)
\right)
\right]-1
\right)dx\mu (dg)
\right).
\end{align}(see for instance \cite{HeiSch})
 Call $\psi _{2}(t,s,v)$ the version with the role of $\u_{1}$ and $\u_{2}$ switched,
\begin{align*}
\psi _{2}(t,s,v)=\exp\left(
\int_{\mathbb{R}^{2}\times \G}\left(
\exp\left[
\imath\left(
tg(x)+(s_{1}\partial _{2}g(x)+s_{2}\partial _{1}g(x))+v\partial^2_{22}g(x)
\right)
\right]-1
\right)dx\mu (dg)
\right),
\end{align*}notice how $s_{1}$ and $s_{2}$ have been switched in the scalar product with $\nabla g(x)$.
 Introduce the reduced forms 
$\psi _{i}(t,s)=\psi _{i}(t,s,0),\psi (t)=\psi _{i}(t,0)$. Under the hypothesis \eqref{eq:decay-g} below, we can deduce from \eqref{eq:FT-SN} an expression of the partial derivatives 
\begin{align*}
\partial _{4}\psi_{1} (t,s ,0)&=\frac{d\psi _{i}(t,s_{1},s_{2},v)}{d  {v}}\Big | _{v=0}\\
&=\imath\psi_{1} (t,s)\int_{\mathbb{R}^{2}\times \G}\partial^2_{11}g(x)\exp\left[
 \imath(tg(x)+(s\cdot \nabla g(x)))
\right]dx\mu (dg)\\
\end{align*}
with a similar expression for $\partial _{4}\psi _{2}(t,s,0)$, and
\begin{align*}
\partial ^{2}_{2,2}\psi_{i} (t,(0,0))&=-\psi (t)\Bigg[
\Big(
\int_{\mathbb{R}^{2}\times \G}\partial _{i}g(x)\exp(\imath tg(x)) dx\mu (dg)\Big)^{2}\\
&\hspace{4cm}+
\int_{\mathbb{R}^{2}\times \G}\partial _{i}g(x)^{2}\exp(\imath tg(x)) dx\mu (dg)
\Bigg].
\end{align*}

\begin{theorem}
\label{thm:stat-SN} 
Assume that the hypotheses of Theorem \ref{thm:EC-SN} prevail, that the measure $\mu $ is isotropic,  and furthermore that for some $\gamma >4$, for $\mu $-a.e. $g\in \G$, some random constant $C_{g}$ with finite 4-th moment satisfies
\begin{align}
\label{eq:decay-g}
 \left(
| g(x)|+\sum_{i=1}^{2} \left[
\partial _{i}g(x)  ^{2}+ | \partial^2_{ii}g(x) | 
\right]
\right) \leqslant C_{g}'(1+\|x\|)^{-\gamma }.
\end{align}
Then 
\begin{align}
\label{eq:gamma-0}&\lim_{n\to \infty }\frac{1}{ | W_{n} | }\E\widehat{\chi _{f_{n}}}(t)=\lim_{n\to \infty }\frac{1}{ | W_{n} | }\E I_{f_{n}}(h^{(t)})=  \E\gamma (0,f,h^{(t)}) \\
\notag&=-\frac{\pi -2}{8\pi }\E\left[
e^{\imath tf(0)}\|\nabla f(0)\|^{2}
\right]-\sum_{i=1}^{2}\E\left[
e^{\imath tf(0)}\mathbf{1}_{\{\nabla f(x)\in Q_{i}\}}\partial^2_{ii}f(0)
\right]\\
\notag=&\sum_{i=1}^{2}\Bigg[\frac{\pi -2}{16\pi }\partial ^{2}_{2,2}\psi _{i}(t,(0,0))-\frac{\partial _{4}\psi_i (t,0,0,0)}{4\imath}\\
\notag&\hspace{1cm}+\frac{1}{2\pi ^{2}}\int_{0}^{\infty }\left[
\int_{0}^{\infty }\frac{\partial _{4}\psi_i (t,s_{1}-s_{2},s_{2},0)-\partial _{4}\psi_i (t,s_{1}+s_{2},-s_{2},0)}{s_{1}s_{2}}ds_{1}
\right]ds_{2}\Bigg].
\end{align}
 \end{theorem}
 
 The proof is at Section \ref{sec:stat-SN}.
 
 \begin{remark} 
 With ergodicity arguments, the previous convergence of expectations can 
probably be turned into an almost sure convergence. 
 \end{remark}
 
 \begin{example}
 \label{ex:g0}
 Let us give a class of probability measures $\mu $ that satisfy the hypotheses of Theorem \ref{thm:stat-SN}. Let first $g_{0}(x),x\in \mathbb{R}^{2}$, be such that \begin{itemize}
 \item $g_{0}$ is of class $\mathcal{C}^{2}$ and Morse above $\th,\th>0,$
\item There is $u_{0}>0$ such that $\{g_{0}\geqslant u\}$ is convex for $0<u\leqslant u_{0}$
\item There is $C_{0}>0,\alpha _{0}\geqslant 1$ such that $\|\nabla g(x)\|\leqslant C_{0} | g(x) | ^{\alpha _{0}/2},x\in \mathbb{R}^{2}$, and 
\begin{align*}
\int_{\mathbb{R}^{2}} | g(x) | ^{\alpha _{0}-1}dx<\infty .
\end{align*}
\item There is $C_{0}'>0$ with finite $4$-th moment$,\gamma >4$ such that $g_{0}(x)$ and its first and second order partial derivatives are bounded by $C'_{0}(1+\|x\|)^{-\gamma }$. 
\end{itemize} For example, $g_{0}(x)=\exp(-\|x\|^{2}), g_{0}(x)=(1+x_{1}^{4}+x_{2}^{4,5})^{-1},x\in \mathbb{R}^{2}$, or $g_{0}(x)=\psi _{0}(\|x\|^{2})$, where $\psi_{0} $ is a non-negative $\mathcal{C}^{2}$    function on $\mathbb{R}_{+}$ that vanishes at $\infty $, and $\psi _{0}'<0$,   fit these requirements. 
Assume also that $f$ as defined in \eqref{eq:def-SN} is a.s. Morse above $\th,\th>0$.
Let then $\theta $ be a uniformly distributed random variable in $[0,2\pi )$, and $M$ a random variable in $\mathbb{R}_{+}$ with finite 4-th moment. Then the law $\mu $ of $Mg^{\theta }$, as defined in Remark \ref{rmk:r-theta}, satisfies the hypotheses of Theorem \ref{thm:stat-SN}. Since these requirements are only about moments or almost sure properties of the sample paths, { appropriate} mixtures of such measures also work.
\end{example}

   \section{Moments of higher order}
 \label{sec:moments}

It is difficult, in general, to give general conditions on $f$ ensuring that $\chi (\{f\geqslant\th  \})$ has finite moments of higher order. For the \EP, moment conditions on the field partial derivatives directly imply the finiteness of the moments of $  \chi_{ f}(h)$. 

\begin{proposition} Let $q\geqslant 1.$
Let $f$ be a $\mathcal{C}^{2}$ random   field, and $h $ a $\C^{1}$ function with compact support. Assume that for some real number $M<\infty $, and $p>1,$ for $i=1,2,$
\begin{align*}
\E \partial_{i}f(x)^{2pq}\leqslant M,\;\E | \partial^2_{ii}f(x) | ^{pq}\leqslant M , x\in \mathbb{R}^{2}.
\end{align*} Then
\begin{align*}
\E |I_{f}(h) | ^{q} \leqslant C_{q} \left(M^{1/p}
(\|h\|+\|h'\|)\int_{\mathbb{R}^{2}}\P(f(x)\in \supp(h))^{1-1/p}dx
\right)^{q}
\end{align*}
 
\end{proposition}

\begin{proof}
We have, using several times H\"older inequality,
\begin{align*}
\E\left[
 | I_{f}(h) | ^{q}
\right]\leqslant &\int_{(\mathbb{R}^{2})^{q}}\E\left[\prod_{k=1}^{q}
\mathbf{1}_{\{f(x_{k})\in \supp(h)\}}\left(
\partial _{i}f(x_{k})^{2 }+ | \partial^2_{ii}f(x_{k}) |  
\right)
\right] dx_{1}\dots dx_{q}\\
\leqslant &q\int_{(\mathbb{R}^{2})^{q}}\prod_{k=1}^{q}\left(
\E\left[
\mathbf{1}_{\{f(x_{k})\in \supp(h)\}}\left(
\partial _{i}f(x_{k})^{2q}+ | \partial^2_{ii}f(x_{k}) | ^{q}
\right)
\right]
\right)^{1/q}dx_{1}\dots dx_{q}\\
= &q\left(
\int_{\mathbb{R}^{2}}\left(
\E\left[
\mathbf{1}_{\{f(x )\in \supp(h)\}}\left(
\partial _{i}f(x )^{2q}+ | \partial^2_{ii}f(x ) | ^{q}
\right)
\right] 
\right)^{1/q}dx
\right)^{q}\\
\leqslant &q\left(
\int_{\mathbb{R}^{2}} 
\left(
\E\left[
\mathbf{1}_{\{f(x)\in \supp(h)\}}
\right]
\right)^{1/p'}
\left(\E 
\left[
\partial _{i}f(x)^{2pq}+ | \partial^2_{ii}f(x) | ^{pq}
\right]
\right)^{1/p}
 dx
\right)^{q}\\
\leqslant &qM^{q/p}\left(
\int_{\mathbb{R}^{2}}
\P(f(x)\in \supp(h))^{1-1/p}dx
\right)^{q}.
\end{align*}
\end{proof}
A staightforward application of this result yields the following corollary.
\begin{corollary} 
Let $f$ be a centered Gaussian field with covariance function 
\begin{align*}
\Sigma (x,y)=\E\left[
f(x)f(y)
\right],x,y\in \mathbb{R}^{2},
\end{align*}that satisfies the condition of Remark \ref{rmorse}.
Assume that for some $\th,\alpha   >0$
\begin{align*}
\int_{\mathbb{R}^{2}}\exp\left(
-\frac{\th^{2}  }{2\Sigma (x,x)}
\right)^{1-\alpha }dx<\infty .
\end{align*}
Then for $q\geqslant 1,$ for every $\mathcal{C}^{1}$ function $h:\mathbb{R}\mapsto \mathbb{R}$ with  support in $[u,\infty )$,
\begin{align*}
\E\left(
\int_{\mathbb{R}}h(u)\chi (\{f\geqslant u\})du
\right)^{q}<\infty .
\end{align*}
\end{corollary}

\section*{Appendix}

\subsection{Proof of Theorem \ref{thm:main}}
\label{sec:main-proof}
 
We recall some set-related notation. For   $A,B\subset \mathbb{R}^{2}$, introduce the Minkowski addition
\begin{align*}
A+B=\{a+b;a\in A,b\in B \}.
\end{align*}
 For a set $A\subset \mathbb{R}^{2}$, and   $r>0$, let 
\begin{align*}
A^{\oplus r}=A+B (0,r)=\{x\in \mathbb{R}^{2}:\text{\rm{d}}(x,A)\leqslant r\},
\end{align*}  where $\text{\rm{d}}$ denotes the Euclidean distance in $\mathbb{R}^{2}$, and $B(0,r)$ is the Euclidean ball centered in $0$ with radius $r$. We also note, for $m=1,2$, and some function $g:A\subset \mathbb{R}^{m}\to \mathbb{R}$ of class $\C^{k},k\geqslant 1,$ and $1\leqslant p\leqslant k$, 
\begin{align*}
\|g^{(p)}\|=\sup_{x\in A}\max_{(i_{1},\dots ,i_{p})\in  \{1,\dots ,m\}^{p}}\left|
\frac{\partial ^{p}g(x)}{\partial _{i_{1}}\dots \partial _{i_{p}}}
\right|,
\end{align*}and note $N_{p}(g)=\max_{0\leqslant i\leqslant p}\|g^{(i)}\|.$

  \begin{proof} [Proof of Theorem \ref{thm:main}]
  Note first that in virtue of formula \eqref{eq:intro-classical-formula} and since $f$ is Morse above $\min(\supp(h))$, its number of critical values over $\supp(h)$ is finite and $\chi _{f}(h)$ is well defined.

     Let $\th  _{\min}=\min(\supp(h))$.
Let $K$ be a compact set contained in $W$'s interior, and containing $\{f\geqslant \th  _{\min}\}$ in its interior. 
Let $A,B_{1},B_{2}$, be three independent real random variables with compact support and $\mathcal{C}^{2}$ density on $\mathbb{R}$. For $\eta >0$, let $f_{\eta  }(x)=f(x)+\eta  (A+B_{1}x_{1}+B_{2}x_{2}),x\in K$. Hence there is $\eta _{0}>0,M,\kappa >0$ such that for $\eta  \leqslant\eta _{0}$, a.s. $\{f_{\eta }\geqslant \th  \}$ is contained in $K$'s interior, $N_{2}(f_{\eta })\leqslant M$, and the density $\varphi_{x,\eta } $ of $(f_{\eta  }(x),\partial _{1}f_{\eta  }(x),\partial _{2}f_{\eta  }(x))$ satisfies $N_{2}(\varphi_{x,\eta }) \leqslant \kappa, x\in K.$

 The main technical aspect of the proof is taken care of in the following lemma.

  \begin{lemma}
  \label{lm:heart-result}
  For each $0<\eta  \leqslant \eta  _{0},$
\begin{align*}
\int_{\mathbb{R}}h(\th  )\E\chi (f_{\eta  }\geqslant \th  )d\th  =\E I_{f_{\eta  }}(h).
\end{align*}
\end{lemma}

   \begin{proof}

  Fix $0<\eta \leqslant \eta _{0}.$
The setup implies in particular that $\|\nabla f_{\eta }(x)\|, | \partial _{ii}f_{\eta }(x) |,x\in K, 1\leqslant i\leqslant 2,$ have uniformly bounded moments of any order, and $(f_{\eta }(x),\partial _{1}f_{\eta }(x),\partial _{2}f_{\eta }(x))$  has a uniformly bounded joint density.
 Therefore, using Theorem 9 from \cite{LacEC2}, for $\th  \geqslant \th  _{\min},$
\begin{align*}
\E\chi (\{f_{\eta }\geqslant \th  \})&=\lim_{\varepsilon \to 0}\varepsilon ^{-2}\int_{K}\E [\delta ^{\varepsilon }(x,f_{\eta },\th  )-\delta ^{-\varepsilon }(x,-f_{\eta },-\th  )]dx,
\end{align*}where, for any $\varepsilon >0,A\subset \mathbb{R}^{2},$ function $g:A^{\oplus \varepsilon } \to \mathbb{R},$
\begin{align*}
\delta ^{\pm \varepsilon }(x,g,\th  )=\mathbf{1}_{\{g(x)\geqslant\th  ,g(x\pm\varepsilon \u_{1})<\th  ,g(x\pm\varepsilon \u_{2}<\th  )\}}, x\in A,u\in \mathbb{R}.
\end{align*}

 In the sequel of the proof,  we set $g=f_{\eta }$  for notational simplification.
 Let $1\leqslant i\leqslant 2$ be fixed.  Let $\varepsilon_{0}>0  $ such that $K^{\oplus \varepsilon _{0}}\subset W$ and for $\eta \leqslant \eta _{0},\{f_{\eta }\geqslant \th_{\min}\}^{\oplus \varepsilon _{0}}\subset K$. Take $ 0<\varepsilon \leqslant \varepsilon _{0}$.   Let us first approximate $f(x+\varepsilon \u_{i})$ by $f(x)+\varepsilon \partial _{i}f(x)+\frac{\varepsilon ^{2}}{2}\partial^2_{ii}f(x)$ for each $x\in K,i=1,2$. 
 We have, using Taylor expansion,  
\begin{align*}
g(x+\varepsilon \u_{i})=g(x)+\varepsilon \partial _{i}g(x)+\frac{\varepsilon ^{2}}{2}\partial^2_{ii}g(\tilde x)
\end{align*}
for some $\tilde x\in [x,x+\varepsilon \u_{i}]$. We introduce the continuity modulus, for $D\subset A,$ 
\begin{align*}
\omega _{\partial ^{2}_{ii} f}(s,D)=\sup_{x,y:\|x-y\|\leqslant s} |\partial ^{2}_{ii}f (x)-\partial _{ii}^{2}f (y) |,\;x,y\in D,s>0,
\end{align*}and recall that $\partial ^{2}_{ii}g=\partial ^{2}_{ii}f_{\eta }=\partial ^{2}_{ii}f$.
We have
\begin{align}
\notag\E\int_{K}&\Bigg | 
\mathbf{1}_{\{g(x)\geqslant\th  ,g(x+\varepsilon \u_{i})<\th   \}}-\mathbf{1}_{\{g(x)\geqslant\th  ,g(x)+\varepsilon \partial _{i}g(x)+\frac{\varepsilon ^{2}}{2}\partial^2_{ii}g(x)<\th  \}}\Bigg | dx\\
\notag\leqslant  &\E\int_{K}\Bigg | 
\mathbf{1}_{\{g(x)\geqslant\th  ,g(x+\varepsilon \u_{i})<\th  ,g(x+\varepsilon \u_{i})+\frac{\varepsilon ^{2}}{2}(\partial^2_{ii}g(\tilde x)-\partial^2_{ii} g(x))>\th  \}}\Bigg | dx\\
\notag&+\E\int_{K}\Bigg | 
\mathbf{1}_{\{g(x)\geqslant\th  ,g(x+\varepsilon \u_{i})>\th  ,g(x+\varepsilon \u_{i})+\frac{\varepsilon ^{2}}{2}(\partial^2_{ii}g(\tilde x)-\partial^2_{ii} g(x))<\th  \}}\Bigg | dx\\
\label{eq:taylor-ineq}&\leqslant2 \int_{K}\P\left(
  | g(x+\varepsilon \u_{i})-\th  | \leqslant \frac{\varepsilon ^{2}}{2} | \partial^2_{ii}g(x)-\partial^2_{ii}g(\tilde x) | \leqslant \varepsilon ^{2}\omega _{\partial^2_{ii}f}(\varepsilon ,K)/2
\right)
dx.
\end{align} 

Since   $\partial^2_{ii}f$ is (uniformly) continuous  on $K$, \begin{align*}
2\int_{K}\P\left(
 | g(x+\varepsilon \u_{i})-\th  | \leqslant \frac{\varepsilon ^{2}}{2}(\omega _{\partial^2_{ii}g }(\varepsilon,D )   )
\right)dx\leqslant \kappa | K | \varepsilon ^{2}( \omega _{\partial^2_{ii}f }(\varepsilon ,D))=o(\varepsilon ^{2}).
\end{align*}
For $u\in \supp(h),x\in K,0<\varepsilon\leqslant \varepsilon _{0}, $ define
\begin{align*}
\delta '^{\eta }(x,f,\th  )=\mathbf{1}_{\{g(x)\geqslant\th  ,g(x)+\varepsilon \partial _{i}g(x)+\frac{\varepsilon ^{2}}{2}\partial^2_{ii}g(x)<\th  ,i=1,2\}},\eta \in \mathbb{R}.
\end{align*} The inequality \eqref{eq:taylor-ineq} applied twice yields, as $\varepsilon \to 0,$
\begin{align*}
\varepsilon ^{-2} 
\int_{K} 
\left[
\delta ^{\varepsilon }(x,g,\th  ) -\delta '^{\varepsilon }(x,g,\th  ) 
\right]dx
 \to 0.
\end{align*} Proving a similar result for $\delta ^{-\varepsilon }(x,-g,-\th  )$, we have for $u\in \supp(h)$
\begin{align}
\label{eq:approx-bicov}
\E\chi (\{g\geqslant\th  \})=\lim_{\varepsilon \to 0}\varepsilon ^{-2}\int_{K}\E\left[
\delta '^{\varepsilon }(x,g,\th  )-\delta '^{-\varepsilon }(x,-g,-\th  )
\right]dx.
\end{align}
 Then,  fix $x\in K$ and denote by $\varphi_{x} $ the density of the random vector $(g(x),\partial _{1}g(x),\partial _{2}g(x)),$ and note $t=\partial _{11}^2g(x)/2=\partial _{11}^2f (x)/2,s=\partial _{22}g(x)/2=\partial _{22}f(x)/2$. We have
\begin{align*}
 \E \delta '^{\varepsilon }(x,g,\th  ) &=\E\mathbf{1}_{\{g(x)\geqslant \th  ,g(x)+\varepsilon\partial _{1}g(x)+ {\varepsilon ^{2}}t<\th  ,g(x)+\varepsilon \partial _{2}g(x)+ {\varepsilon ^{2}}s<\th  \}}dx\\
&=\int_{\mathbb{R}^{3}}\mathbf{1}_{\{a\geqslant \th  ,a+\varepsilon (b+\varepsilon t)<\th  ,a+\varepsilon (c+\varepsilon s)<\th  \}}\varphi_x  (a,b,c )dadbdc \\
&=\int_{\mathbb{R}^{3}}\mathbf{1}_{\{a\geqslant \th  ,a+\varepsilon \beta <\th  ,a+\varepsilon \gamma <\th   \}}\varphi_x  (a,\beta -\varepsilon t,\gamma -\varepsilon s)dad\beta d\gamma \\
&=\int_{\mathbb{R}^{2}}\mathbf{1}_{\{\beta <0,\gamma <0\}}\left[
\int_{\th  }^{\th  -\varepsilon \max(\beta ,\gamma )}\varphi_x  (a,\beta -\varepsilon t,\gamma -\varepsilon s )da
\right]d\beta d\gamma \\
&=\varepsilon \int_{\mathbb{R}^{2}}\mathbf{1}_{\{\beta<0,\gamma <0\}}\left[
\int_{ 0}^{ -\max(\beta ,\gamma )}\varphi_x  (\th  +\varepsilon a,\beta -\varepsilon t,\gamma -\varepsilon s )dad\beta 
\right]d\gamma .
\end{align*} 

Recall that $\supp(\varphi _{x})\subset [-M,M]^{3}$.     After doing a second order Taylor-expansion of $\varphi _{x}$, the remainder term above is,  noting $X=(\th  ,\beta ,\gamma ),Y=(a,-t,-s),$ for $\varepsilon <1,$
    \begin{align*} 
 \Big[\int_{\mathbb{R}^{3}}&\mathbf{1}_{\{ \beta <0,\gamma <0, 0<a  <\min( | \beta  | , | \gamma  | )\}}\big | \varphi_x  (X+\varepsilon Y)-\left[ \varphi_x  (X)+\varepsilon \nabla \varphi_x  (X)\cdot Y\right]\big | dad\beta d\gamma  \Big]
 \\
 &\leqslant  \varepsilon ^{2}\|\varphi_x  ^{(2)}\|\int_{\mathbb{R}^{3}}\mathbf{1}_{\{ | a | \leqslant \min( | \beta  | , | \gamma  | )\}}\mathbf{1}_{\{[X,X+\varepsilon Y]\cap \supp(\varphi _{x})\neq \emptyset \}}dad\beta d\gamma \\
&\leqslant \varepsilon ^{2}\|\varphi_x  ^{(2)}\| \int_{\mathbb{R}^{3}}\mathbf{1}_{\{ | a | \leqslant \min( | \beta  | , | \gamma  | )\}} \mathbf{1}_{\{ | \beta |  -\varepsilon   | t | \leqslant M\}}\mathbf{1}_{\{ | \gamma |  -\varepsilon  | s | \leqslant M\}}dad\beta d\gamma \\
&\leqslant 16\|\varphi_x^{(2)}  \| ( M+\varepsilon ( | t | + | s | ))^{3}\varepsilon ^{2}.
  \end{align*}

 It follows that
 \begin{align*}
\int_{K}& \E \delta '^{\varepsilon }(x,g,\th  ) dx\\
&=\varepsilon\int_{K}\Big[ \int_{  \mathbb{R}_{-}^{2}}\int_{0}^{-\max(\beta ,\gamma )}\left(
\varphi _{x}(\th  ,\beta ,\gamma )+\varepsilon (a,-t,-s)\cdot \nabla \varphi _{x}(\th  ,\beta ,\gamma )
\right)dad\beta d\gamma \Big]dx+o(\varepsilon ^{2}) \\
&=\varepsilon\int_{K}\Big[ \int_{  \mathbb{R}_{-}^{2}}\Big[(-\max(\beta ,\gamma ))\big[
\varphi _{x}(\th  ,\beta ,\gamma  )-\varepsilon [s \partial _{2}\varphi_{x} (\th  ,\beta ,\gamma  )+t \partial _{3}\varphi _{x}(\th  ,\beta ,\gamma )]\big] \\
&\hspace{2cm}+\varepsilon \frac{\max(\beta ,\gamma )^{2}}{2}\partial _{1}\varphi_{x} (\th  ,\beta ,\gamma  )
\Big]\Big] d\beta d\gamma  \Big]dx+o(\varepsilon ^{2}).
\end{align*}
Similar computations yield
\begin{align*}
\int_{K} \E \delta '^{-\varepsilon }&(x,-g,-\th  ) dx\\ 
&=\varepsilon\int_{K}\Big[ \int_{  \mathbb{R}_{-}^{2}}\Big[(-\max(\beta ,\gamma ))\big[
\varphi _{x}(\th  ,\beta ,\gamma  )+\varepsilon [s \partial _{2}\varphi_{x} (\th  ,\beta ,\gamma  )+t \partial _{3}\varphi _{x}(\th  ,\beta ,\gamma )]\big] \\
&\hspace{2cm}-\varepsilon \frac{\max(\beta ,\gamma )^{2}}{2}\partial _{1}\varphi_{x} (\th  ,\beta ,\gamma  )
\Big]\Big] d\beta d\gamma  \Big]dx +o(\varepsilon ^{2}).
\end{align*}
Integrations by parts yield
\begin{align*}
\int_{  \mathbb{R}_{-}^{2}}\max(\beta ,\gamma ) \partial _{2}\varphi_x (\th  ,\beta ,\gamma  )d\beta d\gamma  &=-\int_{ \mathbb{R}_{-}^{2}}\mathbf{1}_{\{\beta >\gamma \}}  \varphi_x (\th  ,\beta ,\gamma  ) d\beta d\gamma  \\
\int_{  \mathbb{R}_{-}^{2}}\max(\beta ,\gamma ) \partial _{3}\varphi_x (\th  ,\beta ,\gamma  )d\beta d\gamma &=-\int_{ \mathbb{R}_{-}^{2}}\mathbf{1}_{\{\gamma >\beta \}}  \varphi_x (\th  ,\beta ,\gamma  ) d\beta d\gamma  \\
\end{align*}
which finally gives, using \eqref{eq:approx-bicov},
\begin{align*}
\E \chi (\{g\geqslant \th  \})=\int_{\mathbb{R}^{2}}\Bigg[ \int_{\mathbb{R}^{2}_{+}}{\min(\beta^{2} ,\gamma^{2} )}
\partial _{1}\varphi _{x}(\th  ,\beta ,\gamma   )d\beta d\gamma &-2\int_{}\mathbf{1}_{\{\beta <\gamma \}}s\varphi _{x}(\th  ,\beta ,\gamma ) d\beta d\gamma \\
&-2\int_{}\mathbf{1}_{\{\gamma  <\beta  \}}t\varphi _{x}(\th  ,\beta ,\gamma  ) d\beta d\gamma .
\Bigg]dx.\end{align*} 

Let us now integrate this expression against $h(\th  )$. We can switch integrations on the right-hand side with Fubini's Theorem because each function involved is bounded with compact support, and then perform an integration by parts in the variable $u$. It yields, substituting $s$ and $t$ with $\partial _{11}^2g(x)/2$ and $\partial _{22}g(x)/2,$
\begin{align*}
\int_{\mathbb{R}}h(\th  )\E \chi (\{g\geqslant\th  \})d\th  =& -\int_{\mathbb{R}\times K\times \mathbb{R}_{-}^{2}}h'(\th  )\min(\beta^{2} ,\gamma ^{2}) \varphi _{x}(\th  ,\beta ,\gamma )d\beta d\gamma d\th  dx\\
&-\int_{\mathbb{R}\times K\times \mathbb{R}_{-}^{2}}\mathbf{1}_{\{\beta <\gamma \}}h(\th  )\partial _{11}^2g(x)\varphi _{x}(\th  ,\beta ,\gamma )d\beta d\gamma d\th   dx\\
&-\int_{\mathbb{R}\times K\times \mathbb{R}_{-}^{2}}\mathbf{1}_{\{ \gamma <\beta \}}h(\th  )\partial _{22}g(x)\varphi _{x}(\th  ,\beta ,\gamma )d\beta d\gamma d\th  dx\\
=&-\E\Bigg[\int_{ K}h'(g(x))\mathbf{1}_{\{\partial _{i}g(x)<0,i=1,2\}}\min(\partial _{1}g(x)^{2},\partial _{2}g(x)^{2})dx\\
&-\sum_{i=1}^{2}\int_{K}\partial^2_{ii}g(x)h(g(x))\mathbf{1}_{\{\nabla g(x)\in Q_{i}\}}dx\Bigg].
\end{align*}
The integration over $K$ can be extended to $W$ because $f_{\eta }^{-1}(\supp(h))\subset \text{\rm{int}}(K)$.
Also, $$\mathbf{1}_{\{\partial _{i}g(x)<0,i=1,2\}}\min(\partial _{1}g(x)^{2},\partial _{2}g(x)^{2})=\mathbf{1}_{\{\partial _{2}g(x)<\partial _{1}g(x)<0\}}\partial _{1}g(x)^{2}+\mathbf{1}_{\{\partial _{1}g(x)<\partial _{2}g(x)<0\}}\partial _{2}g(x)^{2},$$
 which concludes the proof.
\end{proof}

 For $\th \in \mathbb{R}$, $g$ a function Morse over $\th,$   and $x\in \mathbb{R}^{2}$, recall that $x\in \{g\geqslant u\} $ is a critical point of index $k\in \{0,1,2\}$ if $\nabla g(x)=0$ and the Hessian matrix $H_{g}(x)$ has exactly $k$ positive eigenvalues.  Introduce 
\begin{align*}
M _{k}(g,\th  )=\{x:g(x)\geqslant \th \text{\rm{ and $x$ is a critical point of index $k$}}\}
\end{align*} and $\mu _{k}(g,\th  )=\#M_{k}(g,\th  )$.
 \begin{lemma}
 \label{lm:cvg-critical-points}
 Let   $g:W\to \mathbb{R}$ of class $\mathcal{C}^{2}$, and $\th\in \mathbb{R}\setminus V(g)  $  such that $\{g\geqslant \th  \}$ is compact and $g$ is Morse above $\th$. Then there is $\eta (g,\th  )>0$ such that for any $\tilde g$ such that $N_{2}(g -\tilde g ) \leqslant \eta(f,\th  ) ${, then $\mu _{k}(g,\th  )=\mu _{k}(\tilde g,\th  )$ for $k=0,1,2$.  In particular,} $\chi (\{g\geqslant \th  \})=\chi (\{\tilde g\geqslant \th  \})$.
\end{lemma}

\begin{proof} 
We call absolute angle formed by two vectors $s,t\in \mathbb{R}^{2}$, with value in $[0,\pi /2]$, the angle between the lines spanned by $s$ and $t$ in $\mathbb{R}^{2}$. Also note $B^{(2)}(g,\eta )$ the ball with radius $\eta \geqslant 0$ in the space of $\mathcal{C}^{2}$ functions on $W$, endowed with the norm $N_{2}(\cdot ).$

 A compacity argument easily yields that $M _{k}(g,\th  )$ is  finite.
Let $\alpha _{1}>0,\theta \in (0,\pi /2]$ such that for $x\in M _{k}(g,\th  )$,  the absolute angle formed by   $\nabla \partial _{1}g(y)$ and $\nabla \partial _{2}g(y)$,   at any point $y\in B(x,\alpha _{1})$, is larger than $\theta $. The reason why these normals are not collinear is that $\det(H_{g}(x))\neq 0$. 

Since $y\mapsto \nabla \partial _{i}g(y)$ is continous, let $0<\alpha_{2}\leqslant \alpha _{1} $ be such that the angle between $\nabla \partial _{i}g(y)$ and $\nabla \partial _{i}g(y')$ is smaller than $\theta /5$ for $y,y'\in B(x,\alpha_{2} ),x\in M_{k}(g,\th  ),i=1,2$. 
 
  There is $\eta _{1}>0$ depending also on $\alpha_{2} $  such that, for $x\in M_{k}(g,\th  )$, and $\tilde g\in B^{(2)}(g,\eta _{1})$, for every $y\in B(x,\alpha_{2} )$, the absolute angles formed by $\nabla \partial _{i}\tilde g(y)$ and $\nabla\partial _{i}g(y )$  are smaller than $\theta /5$.

  It yields that for $y,y'\in B(x,\alpha_{2} )$, the absolute  angles between $\nabla \partial _{i}\tilde g(y),\nabla \partial _{i}\tilde g(y'),i=1,2,$ are smaller than $3\theta /5$, and the absolute  angle between $\nabla \partial _{1}\tilde g(y)$ and $\nabla \partial _{2}\tilde g(y')$ is larger than $\theta -2\theta /5=3\theta /5$.

Let $i\in \{1,2\}$. There is $ \kappa >0,0<\alpha _{3}\leqslant \alpha _{2}$ such that for $x\in M_{k}(g),y\in B(x,\alpha _{3})$, $\|\nabla \partial _{i}g(y)\|>\kappa $. There is $\eta _{2}>0$ such that for $ \tilde g\in B^{(2)}(g, \eta _{2})$, $\|\nabla \partial _{i}\tilde g(y )\|>2\kappa /3$ on $B(x,\alpha _{3})$, and there is $\eta _{3}>0$ such that for $\tilde g\in B^{(1)}(g,\eta _{3})$,  $ | \partial _{i}\tilde g(x) | \leqslant  \alpha _{3} \kappa /3$ on $B(x,\alpha _{3} )$.

Let $x\in M_{k}(g,\th  ),y\in B(x,\alpha _{3}),i\in \{1,2\}${ and $\tilde g\in B^{(2)}(g,\eta _{3})$}. Assume without loss of generality that $\partial _{i}\tilde g(x)\geqslant 0$ (the reasoning is similar if $\partial _{i}\tilde g(x)\leqslant 0$). Then, for the right choice of $s\in \{1,-1\}$
\begin{align*}
\partial _{i}\tilde g\left(
x-s\alpha _{3} \frac{\nabla \partial _{i}\tilde g(x)}{\|\nabla \partial _{i}\tilde g(x)\|} 
\right)\leqslant &\partial _{i}\tilde g(x)- {\alpha _{3} \min_{y\in B(x,\alpha _{3} )}\|\nabla \partial _{i}\tilde g(y)\|}  \\
&\leqslant \partial _{i}\tilde g(x)-\frac{\alpha _{2}  \kappa }{3} ,
\end{align*}and this quantity is $\leqslant 0$, whence $\partial _{i}\tilde g(x_{i}^{\tilde g })=0$ for some $x_{i}^{\tilde g }\in B(x,\alpha _{3})$, for $i=1,2$.

 Let now $0<\alpha _{4}\leqslant \alpha _{3}$ be such that for any $x\in M_{k}(g,\th  ),$ for any function $\tilde g\in B(g,\eta _{3})$ such that the absolute angles between $\nabla \partial _{1}\tilde g$ and $\nabla \partial _{2}\tilde g$ is larger than $\theta /5>0$ on $B(x,\alpha _{4})$, then if the two $\mathcal{C}^{1}$ manifolds $\{\partial _{1}\tilde g=0\}$ and $\{\partial _{2}\tilde g=0\}$ both touch $ B(x,\alpha _{4})$, then they meet at a unique  point $z_{x,\tilde g}\in B(x,\alpha _{3})$. 
 
   Since the space of non-degenerate matrices with $k$ simple positive eigenvalues is open, and the application which to $\tilde g$ associates $H_{\tilde g}(x)$ is $N_{2}$-continuous, for $\eta _{4}>0$ sufficiently small, for $ \tilde g\in B(g, \eta _{4})$,  for $x\in M_{k}(g,\th  )$, $H_{\tilde g}(z_{x,\tilde g})$ and $H_{g}(x)$  have the same number of negative eigenvalues.

We therefore have shown that for $N_{2}(\tilde g-g)\leqslant \eta (g,\th  )=\eta _{4}$, in the neighborhood of each critical point of $g$, there is one and only one critical point of $\tilde g$, and it has the same index. It remains to prove that for $\eta $ sufficiently small there is no other. Assume that for every $\eta >0$, there is a critical point $z_{\eta  }$ of some $\tilde g_{\eta }\in B^{(2)}(g,\eta )$ such  that $z_{\eta }$ is not in one of the $B(x,\alpha _{4}), x\in M(g,\alpha _{4} )$. By compacity, a subsequence $z_{\eta  '}$ converges to some $ z\in K$, and by continuity $\nabla g(z)=0,g(z)\geqslant \th  $, which implies that $z \in M(x,u)$, reaching a contradiction. 
 We use \eqref{eq:intro-classical-formula} for the conclusion.
\end{proof}
 Let us  finish the proof of Theorem \ref{thm:main}. Recall that by Sard's theorem, any function of class $\mathcal{C}^{1}$ on a domain of $\mathbb{R}^{2}$ has a negligible set of critical values.  
We therefore have, in virtue of Lemma \ref{lm:cvg-critical-points}, for   a.a. $\th  \in \mathbb{R}$, as $\eta \to 0,$
\begin{align*}
\chi (f_{\eta  }\geqslant \th  )\to \chi (f\geqslant \th  ) a.s.
\end{align*}  Lemma \ref{lm:cvg-critical-points} also yields that for $\eta \leqslant \eta (f,\th _{\min} )$, and $\th  \geqslant \th  _{\min}$
\begin{align*}
 | \chi (f_{\eta }\geqslant \th  ) |&= | \mu _{2}(f_{\eta },\th  )-\mu _{1}(f_{\eta },\th  )+\mu _{0}(f,\th  ) | \leqslant  | \mu_{2} (f_{\eta },\th  _{\min}) | + |\mu _{1}(f_{\eta },\th  _{\min}) | + | \mu _{0}(f_{\eta },\th  _{\min}) | \\
 &\leqslant M(f,\th  _{\min}).  
\end{align*}Since this value is deterministic and $h$ has compact support, Lebesgue's theorem  yields
 $
\E\chi _{f_{\eta }}(h)\to  \chi _{f}(h)$
 as $\eta \to 0.$  Also, since we have $\chi _{f_{\eta }}(h)=I_{f_{\eta }}(h)$ by Lemma \ref{lm:heart-result} and
\begin{align*}
\sup_{x\in K,0<\eta\leqslant \eta (f,\th  _{\min}) } | \gamma (x,f_{\eta },h) | \leqslant \sup_{x\in K,0<\eta\leqslant \eta (f,\th  _{\min}) } | \partial _{i}f_{\eta  }(x)^{2} h'(f_{\eta  }(x))|+ | \partial^2_{ii}f(x)h(f_{\eta  }(x)) | <\infty ,  
\end{align*}
we have, invoking again Lebesgue's theorem,
\begin{align*}
\E \chi _{f_{\eta }}(h)=\E I_{f_{\eta }}(h)=\E \int_{\mathbb{R}^{2}}\gamma (x,f_{\eta },h)dx\to \int_{\mathbb{R}^{2}}\E \gamma (x,f,h)=I_{f}(h),
\end{align*}which gives the conclusion.

 \end{proof}
 
 \subsection{Proof of Theorem  \ref{thm:EC-SN}}
\label{sec:proof-SN}

\begin{proof}

According to Theorem \ref{thm:main}, for each test function $h$ supported by $(0,\infty )$, we have a.s. 
\begin{align*}
\chi _{f}(h)=I_{f}(h).
\end{align*}Let $g_{q,p},q,p\geqslant 1$ be a family of function satisfying the following:
\begin{itemize}
\item $ | g_{q,p} | \leqslant 1$ and $g_{q,p}$ is of class $\mathcal{C}^{1}$ with support in $[2^{-p},q+1]$
\item $g_{q,p}(u)=1$ for $u\in [2^{-p+1},q]$  
\item $g_{q,p}$ and $g_{q+1,p}$ coincide on $\mathbb{R}\setminus [q,q+1]$
\item For $u\in \mathbb{R}$, $\left|
g_{q,p}'(u)
\right|\leqslant \mathbf{1}_{\{u\in [2^{-p},2.2^{-p} ]\}}2^{p+1}+2\mathbf{1}_{\{u\in [q,q+1]\}}$.
\end{itemize}   Let $g_{p}(\th  )=\lim_{q\to \infty }g_{q,p}(\th  ),\th>0$,   define 
\begin{align*}
h_{q,p}(\th  )=g_{q,p}(\th  )\exp(\imath \th  t),h_{p}(\th  )=g_{p}(\th  )e^{\imath \th  t}, t\in \mathbb{R},
\end{align*}
and recall that $\chi _{f}(h_{q,p})=I_{f}(h_{q,p})$.

With probability $1$, $f$ has a finite maximum $M$, whence $\chi (f\geqslant\th  )=0$ for $\th  >M $.
From there, it easily follows that $\chi_{f} (h_{q,p})=\chi_{f} (h_{[M]+1,p})$ for $q>M$, and a.s., 
\begin{align*}
\chi_{f} (h_{p})=\lim_{q\to \infty }\chi_{f} (h_{q,p})=\lim_{q\to \infty }I_{f}(h_{q,p}).
\end{align*} 
Then, for fixed $p\geqslant 1$, the support of $h_{q,p}$ is contained in the set $\{f\geqslant \min(h_{p})\}$, which is compact since $f(x)\to 0$ as $\|x\|\to \infty .$
The integrand of $I_{f}(h_{q,p})$ is a.s. bounded by the continuous function
\begin{align*}x\mapsto \left(
\|h_{ p}\|_{\infty }+\|h_{ p}'\|_{\infty }
\right) \sum_{i=1}^{2}\left[
\partial _{i}f(x)^{2}+\left|
\partial^2_{ii}f(x)
\right|
\right].
\end{align*} 
 Therefore Lebesgue's theorem yields $\chi_{f} (h_{p})=I_{f}(h_{p})$ for $p\geqslant 1$.

Let us deal with the singularity of $f$ around $0$.  Call $N$ the number of points of $\tilde\eta $, and call $\{(g_{k},y_{k}),1\leqslant k\leqslant N\}$ the points of $\tilde\eta $. By Assumption \ref{ass:g-structural}, $\mu $-almost every function $g$ has a threshold $\th   _{g }>0$ such that $\chi (f\geqslant \th  )$ is convex for $\th  \leqslant \th  _{g}$.   Therefore, for $\th  \leqslant \th  _{f}:= \min(\th  _{g_{k}},k=1,\dots ,N)$, $\{f\geqslant \th  \}$ is the union of at most $N$ convex sets, its  \EC~is therefore bounded by $N^{2}$ (this can be proved by induction by considering the number of connected components of the excursion set, or of its complement). As a result, we have the a.s. convergence
\begin{align*}
\chi _{f}(h^{(t)})=\lim_{p\to \infty }\chi _{f}(h_{p})=\lim_{p\to \infty }I_{f}(h_{p}).
\end{align*} 
The last convergence is more delicate. Developping $h_{p}'$ in $I_{f}(h_{p})$ yields
\begin{align*}
I_{f}(h_{p})=&-\sum_{i=1}^{2}\int_{\mathbb{R}^{2}}\mathbf{1}_{\{\nabla f(x)\in Q_{i}\}}\left[
 (g_{p}h^{(t)})'(f(x )) \partial _{i}f(x)^{2}+ g_{p}(f(x))\exp(\imath   tf(x)) \partial^2_{ii}f(x)
\right]dx\\
=&-\sum_{i=1}^{2}\int_{\mathbb{R}^{2}}\mathbf{1}_{\{\nabla f(x)\in Q_{i}\}}\left[
\Big[
 (g_{p}'(f(x) ))+itg_{p}(f(x))
\right]\exp(\imath tf(x)  ))\partial _{i}f(x)^{2}\\
&\hspace{5cm}+ g_{p}(f(x) )\exp(\imath tf(x) ) \partial^2_{ii}f(x)
\Big]dx\\
\end{align*}

To estimate these three terms, recall that $ | g_{p}(u) | \leqslant 1$ and that $ | g_{p}'(u) | \leqslant 2^{p+1}\mathbf{1}_{\{2^{-p}\leqslant u\leqslant 2^{-p+1}\}},u\geqslant0$. Using \eqref{eq:integrability-g}, the second and third terms converge a.s. to 
the right hand side of \eqref{eq:EC-SN-bounded}.

To prove that  the first term vanishes, take $x\in \mathbb{R}^{2}$. We have that a.s., as $p\to \infty ,$
\begin{align}
\label{eq:as-convg-gp}
2^{p}\mathbf{1}_{\{ 2^{-p}\leqslant f(x)\leqslant 2.2^{-p}\}}\to 0.
\end{align} 
 Let us prove that we have the domination, for $i=1,2$, 
\begin{align}
\label{eq:domination}
\int_{\mathbb{R}^{2}}\sup_{p\geqslant 1}2^{p} \left[
\partial _{i}f(x)^{2}\mathbf{1}_{\{2^{-p}\leqslant f(x)\leqslant 2.2^{-p}\}}
\right]dx<\infty .
\end{align} 
The non-negativity of the $g_{k}$ and \eqref{eq:def-SN} yield that for $x\in \mathbb{R}^{2},p\geqslant 1,1\leqslant k\leqslant N$, $\mathbf{1}_{\{  f(x)\leqslant 2.2^{-p}\}}\leqslant \mathbf{1}_{\{   g_{k}(y_{k}-x)\leqslant 2.2^{-p}\}}$.
The left-hand side  of \eqref{eq:domination} is bounded by   \begin{align*}
& \int_{\mathbb{R}^{2}}
 \sum_{i=1}^{2}\sum_{k,l=1 }^{N} \left[\partial _{i}g_{k}(y_{k}-x) \partial _{i}g_{l}(y_{l}-x)\sup_{p\geqslant 1}2^{p}\left[
\mathbf{1}_{\{ g_{k}(y_{k}-x)\leqslant 2.2^{-p}\}}
\right] 
\right]dx\\ 
 \leqslant &2\int_{\mathbb{R}^{2}}
 \sum_{i=1}^{2} \sum_{k,l=1 }^{N} \left[\left(
\partial _{i}g_{k}(y_{k}-x) ^{2}+\partial _{i}g_{l}(y_{l}-x)^{2}
\right)\sup_{p\geqslant 1}2^{p}\left[
\mathbf{1}_{\{  g_{k}(y_{k}-x)\leqslant 2.2^{-p}\}}
\right] 
\right]dx\\ 
\leqslant &4\int_{\mathbb{R}^{2}}
 N \sum_{i=1}^{2}\sum_{k=1 }^{N} \left[ \partial _{i}g_{k}(y_{k}-x) ^{2}\sup_{p\geqslant 1}2^{p} \left[
\mathbf{1}_{\{ g_{k}(y_{k}-x)\leqslant 2.2^{-p}\}}
\right] 
\right]dx\\ 
\leqslant& 4N\sum_{i=1}^{2}\sum_{k=1}^{N}C_{g_{k}}\int_{\mathbb{R}^{2}}\sup_{p\geqslant 1}  \mathbf{1}_{\{g_{k}(y_{k}-x)\neq 0\}}g_{k}(y_{k}-x)^{\alpha _{g}}2^{p}\mathbf{1}_{\{2^{p}\leqslant 2 /g_{k}(y_{k}-x)\}}du\\
\leqslant& 8N\sum_{i=1}^{2}\sum_{k=1}^{N}C_{g_{k}}\int_{\mathbb{R}^{2}} g_{k}(z)^{\alpha_{g}-1 } dz<\infty,
\end{align*}by Assumption \ref{ass:g-structural}.
Therefore, \eqref{eq:as-convg-gp} and Lebesgue's theorem yield the conclusion.
\end{proof}

\subsection{Proof of Theorem \ref{thm:stat-SN}}
\label{sec:stat-SN}

 \begin{proof}
According to Theorem \ref{thm:EC-SN}, $\widehat{\chi _{f_{n}}}(t) =I_{f_{n}}(h^{(t)})$ a.s..  The decay hypothesis \eqref{eq:decay-g}   yields that for $1\leqslant q\leqslant 4,$
\begin{align*}
\E \int_{\mathbb{R}^{2} } | g(x) | ^{q}dx\mu (dg)\leqslant \E[C_{g}']^{q}\int_{\mathbb{R}^{2}}(1+\|x\|)^{-q\gamma }dx<\infty ,
\end{align*}which yields that $f(0)$ has finite $4$-th moment. For similar reasons, $\partial _{\u}f(0),\u\in \S^{1},\partial^2_{ii}f(0)$ also have finite $4$-th moment.
We have, using Mecke's formula for the first and second-order moments of a simple Poisson integral (see for instance \cite{HeiSch}),
\begin{align}
\notag
\notag\E&\left|
 I_{f _{n}}(h)-I_{f_{n}}(h,W_{n})
\right|\leqslant \sum_{i=1}^{2}N_1(h) \int_{W_{n}^{c}}\E\left[ | \partial^2_{ii}f_{n}(x) | +
\partial _{i}f_{n}(x)^{2}
\right]dx\\
\notag\leqslant &N_1(h) \sum_{i=1}^{2}\int_{W_{n}^{c}}\Big[\int_{W_{n}} | \partial^2_{ii}g(y-x) | dy\mu (dg)\\
\notag&\hspace{5cm}+
\Big[
\int_{ W_{n}}\partial _{i}g(y-x)dy\mu (dg)
\Big]^{2}+\int_{W_{n}}\partial _{i}g(y-x)^{2}\mu (dg)dy
\Big]dx\\
\notag\leqslant &N_1(h) \int_{W_{n}^{c}}\Bigg[
\E_{\mu }[ C_{g}']\int_{W_{n}}(1+\|y-x\|)^{-\gamma }dy+\E_{\mu } [(C'_{g})^{2}]\int_{W_{n}}(1+\|y-x\|)^{-2\gamma }dy \\
\label{eq:1t-bound-fn}&\hspace{7cm}+\E_{\mu } [(C_{g}') ^{2}]\left[
\int_{ W_{n}} (1+\|y-x\|)^{-\gamma }dy 
\right]^{2}
\Bigg]dx.
 \end{align} Up to a constant, the previous quantity can be bounded by the same expression where $\|\cdot \|$ is replaced by the norm $\|y-x\|_{1}= | y_{1}-x_{1} | + | y_{2}-x_{2} | ,x,y\in \mathbb{R}^{2}.$
 Let $x\in W_{n}^{c}$. Assume that $x$ is on the positive horizontal axis, i.e. that the coordinates of $x$ are $(\|x\|,0)$.
  It is clear that $W_{n}\subset [-\infty ,\sqrt{n}]\times \mathbb{R}$. We have 
\begin{align*}
\int_{W_{n}}(1+\|y-x\|_{1})^{-\gamma }dy\leqslant& \int_{-\infty }^{\sqrt{n}}\int_{-\infty }^{\infty }(1+\|x\|-l+ | t | )^{-\gamma }dldt\\
=&\int_{-\infty }^{\sqrt{n}}(1+\|x\|-l)^{-\gamma }\int_{\mathbb{R}}\left(
1+\frac{ | t | }{1+\|x\|-l}
\right)^{-\gamma }dtdl\\
=&\int_{-\infty }^{\sqrt{n}}(1+\|x\|-l)^{1-\gamma }\int_{\mathbb{R}}(1+ | t | )^{-\gamma }dtdl\\
\leqslant &c_{\gamma }\left[
\frac{(1+\|x\|-l)^{2-\gamma }}{2-\gamma }
\right]_{-\infty }^{\sqrt{n}}=c'_{\gamma }(1+\|x\|-\sqrt{n})^{2-\gamma }.
\end{align*}
It follows that, after a polar change of coordinates, up to a multiplicative constant, \eqref{eq:1t-bound-fn} is bounded by 
\begin{align*}
2\pi  \int_{\sqrt{n}}^{\infty }&\left(
(1+r-\sqrt{n})^{2-\gamma }+(1+r-\sqrt{n})^{2-2\gamma }+(1+r-\sqrt{n})^{4-2\gamma }
\right)rdr =O(\sqrt{n}),
\end{align*}using $\gamma >4$. Therefore $I_{f_{n}}(h)-I_{f_{n}}(h,W_{n})=o( | W_{n} | )$.

  Let $W_{n}'={   B(0,\sqrt{n}-n^{1/4}) } $. 
The very expression of $I_{f_{n}}(h,W_{n})$ yields that $\E I_{f_{n}}(h,W_{n})-\E I_{f_{n}}(h,W_{n}')=o( | W_{n} | ).$ The rest of the proof consists in showing that $I_{f_{n}}(h,W_{n}')-I_{f}(h,W_{n}')=o( | W_{n} | )$. Put $g_{n}=f_{n}-f$. As for $f_{n}$, $g_{n}$ and its partial directional derivatives have finite moments up to order $4$.    Corollary \ref{cor:continuity} yields
\begin{align}
\notag
\E&\left|
I_{f_{n}}(h,W_{n}')-I_{f}(h,W_{n}')
\right| \leqslant c  
N_2(h)
   \max_{i}\left(
\E\partial _{i}f(0)^{4 },\E | \partial^2_{ii}f(0) | ^{2},2\E  \partial _{i}f(0)^{2 }  +\E [(C_{g}')^{2}] \right)^{1/2} \\
\label{eq:intermed}& \times \int_{W_{n}'}\max_{i}\left(
 \E | g_{n}(x) |^{2 }  ,\E |  \partial _{i}g_{n}(x)  | ^{2},\E | \partial^2_{ii}g_{n}(x) |^{2 } 
  , \E\delta _{x}(f,g_{n})\right) ^{1/2 }dx.
\end{align}Note $\tilde\eta _{\infty }$ a Poisson measure with intensity $\ell\times \mu $ on $\mathbb{R}^{2}\times \G$. It can also be built as $\tilde\eta _{\infty }=\lim_{n\to \infty }\tilde \eta _{n} $, where the limit is a pointwise convergence on each compact.
  We have, using again Mecke's formula for the second order moment ((3.1) in \cite{HeiSch}),
\begin{align*}
\int_{W_{n}'}\E &g_{n}(x)^{2}dx\leqslant \int_{W_{n}'}\left[
\sum_{y\in \eta_{\infty } \cap W_{n}^{c}}(C_{g_{y}}')^{2}(1+\|y-x\|^{-2\gamma })
\right]^{2}dx\\
=&\int_{W_{n}'}\left[
\left(
\int_{W_{n}^{c}}C_{g}'(1+\|y-x\|)^{-2\gamma }dy\mu (d\gamma )
\right)^{2}+\int_{W_{n}^{c}}C_{g}'(1+\|y-x\|)^{-2\gamma }dy\mu (dg)
\right]dx\\
\leqslant &\int_{W_{n}'}\left[
\left(
\int_{B(x,n^{1/4})^{c}}C_{g}'(1+\|y-x\|)^{-2\gamma }dy\mu (d\gamma )
\right)^{2}+\int_{B(x,n^{1/4})^{c}}C_{g}'(1+\|y-x\|)^{-2\gamma }dy\mu (dg)
\right]dx\\
\leqslant &c | W_{n}' |\left[
\left(
 \int_{n^{1/4}}^{\infty }(1+r)^{-2\gamma }rdr
\right)^{2}+\int_{n^{1/4}}^{\infty }(1+r)^{-4\gamma }rdr
\right]\\
\leqslant& c' | W_{n}' | \left(
(n^{(-2\gamma +2)/4})^{2}+n^{(-4\gamma +2)/4}
\right)\leqslant c''| W_{n}' |n^{-\gamma +1}=o( | W_{n}' | )=o( | W_{n} | ).
\end{align*}
A similar bound holds for $\int_{W_{n}'}\partial ^{2}_{ii}g_{n}(x)^{2}dx$ and $\int_{W_{n}'}  \partial _{i}g_{n}(x)^{2}dx.$
 
To complete bounding \eqref{eq:intermed}, let now $x\in W_{n}',\u\in \S^{1}$.  We have 
\begin{align*}
\P( | \partial _{\u}f (x) | \leqslant  | \partial _{\u}g_{n}(x))\leqslant & \P\left( | \partial _{\u}  f(0) | \leqslant    \sum_{y\in \eta _{\infty }\cap B(x,n^{1/4})^{c}}C_{g_{y}} (1+\|x-y\|)^{-2}   \right)
\end{align*}
Since $\E _{\mu }C_{g}<\infty $ and $\int_{\mathbb{R}^{2}}(1+\|y\|)^{-2\gamma }dy<\infty $, $\E \int_{\mathbb{R}^{2}}\sum_{y\in \eta _{\infty }}C_{g_{y}}(1+\|y\|^{-2})<\infty $, and $\sum_{y\in \eta _{\infty }}C_{g_{y}}(1+\|y\|)^{-2\gamma }<\infty $ a.s.. Lebesgue's theorem theorefore yields that a.s. $\sum_{y\in \eta _{\infty }\cap B(x,n^{1/4})^{c}}C_{g_{y}}(1+\|y\|)^{-2}\to 0$ a.s. as $n\to \infty $. It follows that
 $
 \P( | \partial _{\u}f (x) | \leqslant   | \partial _{\u}g_{n}(x))$ converges to $\P(\partial _{\u}f(0)=0)$ uniformly in $x$.

By isotropy, $\P(\partial _{\u}f(0)=0)$ does not depend on $\u $, and can therefore not be non-zero by $\sigma $-finiteness of the law of $\|\nabla f(0)\|^{-1}\nabla f(0)$. 
Finally
\begin{align*}
\left|
\E I_{f_{n}}(h,W_{n})-\E I_{f}(h,W_{n})
\right|&\leqslant o( | W_{n} | )
\end{align*}
and \eqref{eq:gamma-0} is proved.

According to Remark \ref{rmk-isotropy}, the first term in \eqref{eq:gamma-0} is 
\begin{align*}
\frac{\pi -2}{16\pi }\E \exp(\imath tf(0))\|\nabla f(0)\|^{2}=\frac{\pi -2}{16\pi}\E \left[
\exp(\imath tf(0))(\partial _{1}f(0)^{2}+\partial _{2}f(0)^{2})
\right]
\end{align*}
and, since $\E \partial _{i}f(0)^{2}<\infty ,$
\begin{align*}
\E\exp(\imath tf(0))\partial _{i}f(0)^{2}=-\frac{d^{2}}{ds ^{2}}\Big | _{s =0}\E\exp(\imath (tf(0)+s\partial _{i}f(0)))=-\partial ^{2}_{2,2}\psi _{i}(t,(0,0)).
\end{align*}
It follows that the first term of  \eqref{eq:gamma-0} is 
\begin{align*}
\frac{\pi -2}{16\pi}\sum_{i=1}^{2}\partial ^{2}_{2,2}\psi (t,0,0).
\end{align*}
Let us now take care of the second term. Start with the summand $i=1$. Recalling that $\mathbf{1}_{\{\nabla f(x)\in Q_{1}\}}=\mathbf{1}_{\{\partial _{2}f(0)<\partial _{1}f(0)<0\}}$, we have
\begin{align*}
 \mathbf{1}_{\{\partial _{2}f(0)<\partial _{1}f(0)\}}\mathbf{1}_{\{\partial _{1}f(0)<0\}} =\frac{1}{4} 
(1-\sign(\partial _{2}f(0)-\partial _{1}f(0)))(1-\sign(\partial _{1}f(0)))
.
\end{align*}
The isotropy and stationarity of $f$ entail  that in any point $x$, $(f(x+y),y\in \mathbb{R}^{d})\equlaw (f(x-y),y\in \mathbb{R}^{d})$, whence 
\begin{align}
\label{eq:symmetry-nabla}
(f(x),\nabla f(x),\partial^2_{ii}f(x))\equlaw (f(x),-\nabla f(x),\partial^2_{ii}f(x)).
\end{align} As a consequence, 
\begin{align*}
\E\left[
\exp(\imath tf(0)) \partial _{11}^2f(0)\sign(\partial _{1}f(0))
\right]=-\E\left[
\exp(\imath tf(0))\partial _{11}^2f(0)\sign(\partial _{1}f(0))
\right]=0,
\end{align*}and similarly $\E\exp(\imath tf(0))\partial _{11}^2f(0)\sign(\partial _{2}f(0)-\partial _{1}f(0))=0$. We end up having to compute 
\begin{align}
\label{eq:symmetry} 
\frac{1}{4}
\E\left[
\exp(\imath tf(0)) \partial _{11}^2f(0)
 \left(1+
\sign(\partial _{2}f(0)-\partial _{1}f(0))\sign(\partial _{1}f(0))
\right)\right].
\end{align}

 Using the fact that for any $w\in \mathbb{R}$ we have the improper integral
\begin{align*}
\lim_{X\to \infty }\int_{0}^{X}\frac{\sin(uw)}{u}du=\sign(w)\frac{\pi }{2},
\end{align*}
we have for any random variables  $U,V,W$ such that $\E | U | <\infty $
\begin{align}
\label{eq:sign-sin}
\E\left[
U\sign(V)\sign(W)\right]=\frac{4}{\pi ^{2}}\lim_{X\to \infty }\left[
\lim_{Y\to \infty }\int_{0}^{X}\left[
\int_{0}^{Y}\E\left[
U\frac{\sin(Ws_{1})}{s_{1}}\frac{\sin(Vs_{2})}{s_{2}}
\right]ds_{1}
\right]ds_{2}
\right].
\end{align}

Using \eqref{eq:symmetry-nabla} again, the characteristic function  satisfies $\psi _{i}(t,s,v)=\psi_{i} (t,-s,v),t,v\in \mathbb{R},s\in \mathbb{R}^{2},i=1,2$. The equality still holds after derivation with respect to the first or third argument. Below, at the third line, we use $\partial _{v}\psi ^{i}_x(t,s,v)=\partial _{v}\psi ^{i}_x(t,-s,v)$.  

We have, using \eqref{eq:sign-sin} and \eqref{eq:symmetry}  
 \begin{align*}
\E &\left[
\exp(\imath tf(0))\partial _{11}^2f(0)\mathbf{1}_{\{\nabla f(0)\in Q_{1}\}}
\right]
 \\
&=\frac{1}{4}\Bigg[\E e^{\imath tf(0)} \partial _{11}^2f(0)\\
&\hspace{1cm}+\frac{4}{\pi ^{2}}\lim_{X,Y}\int_{0}^{X}\int_{0}^{Y}\E\left[e^{\imath tf(0)}
\partial _{11}^2f(0)\frac{\sin(s_{1}\partial _{1}f(0))}{s_{1}}\frac{\sin(s_{2}(\partial _{2}f(0)-\partial _{1}f(0)))}{s_{2}}
\right]ds_{1}ds_{2}\Bigg]\\
&= 
\frac{1}{4\imath}\frac{d }{d v}\big | _{v=0}\psi_1 (t,0,0,v)+\frac{1}{\pi ^{2}}
\int_{0}^{\infty }\Big[\int_{0}^{\infty }\E\big[\partial _{11}^2f(0)e^{\imath tf(0)}\\
&\hspace{3cm}\times \frac{(e^{\imath s_{1}\partial _{1}f(0)}-e^{-is_{1}\partial _{1}f(0)})(e^{\imath s_{2}(\partial _{2}f(0)-\partial _{1}f(0))}-e^{-is_{2}(\partial _{2}f(0)-\partial _{1}f(0))})}{-4s_{1}s_{2}}\Big]ds_{1}\Big]ds_{2}
 \\
&= 
\frac{1}{4\imath}\partial _{4}\psi_1 (u,0,0,0) -\frac{1}{4\pi ^{2}}
\int_{0}^{\infty }\int_{0}^{\infty }\frac{-\imath}{s_{1}s_{2}}\big[\partial _{4}\psi_1 (t,s_{1}-s_{2},s_{2},0)+\partial _{4}\psi_1 (t,s_{2}-s_{1},-s_{2},0)\\
&\hspace{6cm}-\partial _{4}\psi_1 (t,-(s_{1}+s_{2}),s_{2},0)-\partial _{4}\psi_1 (0,s_{1}+s_{2},-s_{2},0)\big]ds_{1}ds_{2}
\Bigg)\\
&=\frac{\partial _{4}\psi_1 (u,0,0,0)}{4\imath}+\frac{\imath}{2\pi ^{2}}\int_{0}^{\infty }\left[
\int_{0}^{\infty }\frac{\partial _{4}\psi_1 (t,s_{1}-s_{2},s_{2},0)-\partial _{4}\psi_1 (t,s_{1}+s_{2},-s_{2},0)}{s_{1}s_{2}}ds_{2}
\right]ds_{1} 
\end{align*}
which we report in 
  \eqref{eq:gamma-0}, with a similar expression for $i=2.$

\end{proof}

 \bibliographystyle{plain}
    \bibliography{bicovariograms}

\begin{thebibliography}{10}

\bibitem{ABBSW}
R.~J. Adler, O.~Bobrowski, M.~S. Borman, E.~Subag, and S.~Weinberger.
\newblock Persistent homology for random fields and complexes.
\newblock {\em IMS Coll.}, 6:124--143, 2010.

\bibitem{AdlSam}
R.~J. Adler and G.~Samorodnitsky.
\newblock Climbing down {G}aussian peaks.
\newblock to appear in Ann. Prob.,
  \href{http://arxiv.org/abs/1501.07151}{\texttt{preprint arXiv 1510.07151}},
  2015.

\bibitem{AST}
R.~J. Adler, G.~Samorodnitsky, and J.~E. Taylor.
\newblock High level excursion set geometry for non-gaussian infinitely
  divisible random fields.
\newblock {\em Ann. Prob.}, 41(1):134--169, 2013.

\bibitem{AdlTay07}
R.~J. Adler and J.~E. Taylor.
\newblock {\em Random Fields and Geometry}.
\newblock Springer, 2007.

\bibitem{BieDes12}
H.~Bierm\'e and A.~Desolneux.
\newblock Crossings of smooth shot noise processes.
\newblock {\em Ann. Appl. Prob.}, 22(6):2240--2281, 2012.

\bibitem{BieDes}
H.~Bierm\'e and A.~Desolneux.
\newblock On the perimeter of excursion sets of shot noise random fields.
\newblock {\em Ann. Prob.}, 44(1):521--543, 2016.

\bibitem{EstLeo14}
A.~Estrade and J.~R. Leon.
\newblock A central limit theorem for the {E}uler characteristic of a
  {G}aussian excursion set.
\newblock to appear in Ann. Prob.,
  \href{https://hal.archives-ouvertes.fr/hal-00943054/}{\texttt{https://hal.archives-ouvertes.fr/hal-00943054/}},
  2014.

\bibitem{HeiSch}
L.~Heinrich and V.~Schmidt.
\newblock Normal convergence of multidimensional shot noise and rates of this
  convergence.
\newblock {\em Adv. Appl. Prob.}, 17(4):709--730, 1985.

\bibitem{KilFri}
J.~M. Kilner and K.~J. Friston.
\newblock Topological inference for {EEG} and {MEG}.
\newblock {\em Ann. Appl. Stat.}, 4(3):1272--1290, 2010.

\bibitem{LacEC1}
R.~Lachi\`eze-Rey.
\newblock Covariograms and {E}uler characteristic {I}. {R}egular sets.
\newblock \href{http://arxiv.org/abs/1510.00501}{\texttt{preprint arXiv
  1510.00501}}, 2015.

\bibitem{LacEC2}
R.~Lachi\`eze-Rey.
\newblock Covariograms and {E}uler characteristic {II}. {R}andom fields
  excursions.
\newblock \href{http://arxiv.org/abs/1510.00502}{\texttt{preprint arXiv
  1510.00502}}, 2015.

\bibitem{LastChapter}
G.~Last.
\newblock {\em in {S}tochastic analysis for {P}oisson point processes:
  {M}alliavin calculus, {W}iener-{I}t\^ochaos expansions and stochastic
  geometry, Ed. by G. Peccati and M. Reitzner}, chapter Stochastic analysis for
  {P}oisson processes.
\newblock Springer International Publishing, Switzerland, 2016.

\bibitem{Mar15}
D.~Marinucci.
\newblock Fluctuations of the {E}uler-{P}oincar{\'e} characteristic for random
  spherical harmonics.
\newblock to appear in Proc. AMS,
  \href{http://arxiv.org/abs/1504.01868}{\texttt{preprint arXiv 1504.01868}},
  2015.

\bibitem{Mel90}
A.~L. Melott.
\newblock The topology of large-scale structure in the universe.
\newblock {\em Physics Reports}, 193(1):1 -- 39, 1990.

\bibitem{SchWei}
R.~Schneider and W.~Weil.
\newblock {\em Stochastic and {I}ntegral {G}eometry}.
\newblock Probability and its Applications. Springer-Verlag, Berlin, 2008.

\bibitem{SWRS}
C.~Scholz, F.~Wirner, J.~G\"otz, U.~R\"ude, G.E. Schr\"oder-Turk, K.~Mecke, and
  C.~Bechinger.
\newblock Permeability of porous materials determined from the {E}uler
  characteristic.
\newblock {\em Phys. Rev. Lett.}, 109(5), 2012.

\bibitem{TayWor}
J.~E. Taylor and K.~J. Worsley.
\newblock Random fields of multivariate test statistics, with applications to
  shape analysis.
\newblock {\em Ann. Stat.}, 36(1):1--27, 2008.

\end{thebibliography}

  \end{document}